\documentclass[final,3p,times]{elsarticle}

\usepackage{graphicx}%
\usepackage{multirow}%
\usepackage{amsmath,amssymb,amsfonts}%
\usepackage{amsthm}%
\usepackage{mathrsfs}%
\usepackage[title]{appendix}%
\usepackage{xcolor}%
\usepackage{textcomp}%
\usepackage{manyfoot}%
\usepackage{booktabs}%
\usepackage{algorithm}%
\usepackage{algorithmicx}%
\usepackage{algpseudocode}%
\usepackage{listings}%
\usepackage{xcolor}
\usepackage{subcaption}

\usepackage{tikz}
\usepackage{pgfplots}
\pgfplotsset{compat=newest}

\theoremstyle{definition}
\newtheorem{definition}{Definition}

\theoremstyle{plain}
\newtheorem{theorem}{Theorem}[section]
\newtheorem{lemma}{Lemma}[section]

\newtheorem{proposition}[theorem]{Proposition}

\newtheorem{example}{Example}

\numberwithin{equation}{section}

\usepackage[english]{babel}

\journal{Applied Numerical Mathematics}

\begin{document}

\begin{frontmatter}

\title{Fractional backward spectral approximation theory for weakly singular adjoint integral equations}

\author[f4]{Mahmoud A. Zaky$^{*,}$}
  \ead{ma.zaky@yahoo.com;mibrahimm@imamu.edu.sa}

\address[f4]{Department of Mathematics and Statistics, College of Science, Imam Mohammad Ibn Saud Islamic University (IMSIU), Riyadh, Saudi Arabia}

\cortext[mycorrespondingauthor]{Corresponding author: M.A. Zaky}

\begin{abstract} We introduce a new class of fractional backward orthogonal functions designed for the spectral approximation of weakly singular adjoint Volterra integral equations. These basis functions generate an approximation space that naturally reflects the terminal-endpoint singular behaviour produced by weakly singular kernels. We develop the basic approximation theory for the proposed backward orthogonal basis, including weighted projection estimates, Gauss-type interpolation estimates, inverse inequalities, and stability bounds for the associated weakly singular adjoint integral operator. The error analysis and numerical results show that the proposed backward Jacobi method is particularly suitable for solutions with terminal-endpoint weak singularities  and can recover high-order convergence rates that are typically lost when
usual polynomial approximations are applied directly to such weakly regular
solutions. \end{abstract}
\begin{keyword}
fractional orthogonal functions \sep
adjoint Volterra integral equations \sep
weakly singular kernels \sep
spectral approximation \sep
weighted Sobolev spaces \sep
terminal-endpoint singularities

\end{keyword}

\end{frontmatter}

\section{Introduction}

Orthogonal polynomial approximations are central tools in modern numerical
analysis, owing to their role in spectral methods, high-order interpolation,
quadrature theory, and approximation in weighted Sobolev spaces \cite{shen2011spectral,szeg1939orthogonal}. Classical polynomial
systems such as Legendre, Chebyshev, and Jacobi bases have been extensively
studied and successfully employed in the numerical treatment of differential,
integral, and integro-differential equations \cite{doha2019spectral}. Their
remarkable approximation properties for smooth functions, together with precise
projection and interpolation estimates, have made them indispensable tools in
the development of high-order numerical schemes \cite{diethelm2020good}. Despite this success, classical smooth polynomial bases often encounter
substantial difficulties when applied to functions possessing endpoint
singularities or limited regularity. In many important applications, including
terminal value problems, weakly singular Volterra integral equations, adjoint
Volterra integral equations, and fractional differential equations, the exact
solution typically exhibits reduced smoothness near one endpoint of the
computational interval \cite{ameen2021singularity}. Such singular behaviour
frequently leads to a severe deterioration in the convergence rate of standard
polynomial approximations, even when the global smoothness is only mildly
violated. This issue becomes particularly pronounced in spectral and
collocation methods, where the expected high-order or spectral convergence is
strongly tied to the regularity of the solution.

The numerical treatment of endpoint-singular functions has therefore attracted
considerable attention in recent years \cite{zaky2026new}. Various
singularity-resolving strategies have been proposed, including graded meshes,
singular basis enrichment, coordinate transformations, adaptive polynomial
spaces, and weighted approximation frameworks \cite{zayernouri1945spectral}.
Among these approaches, variable transformations and singularity-adapted basis
constructions have proved particularly effective, since they preserve the global
high-order structure of spectral approximations while incorporating the
intrinsic asymptotic behaviour of weakly regular solutions. In the context of fractional differential equations and weakly singular models,
several important developments have recently been made. In particular,
generalized Jacobi-type functions and transformed orthogonal systems have been
introduced to better capture one-sided singular structures and to restore
spectral accuracy for non-smooth solutions \cite{chen2016generalized}. These
developments demonstrate that suitably adapted basis systems can significantly
outperform classical smooth polynomial approximations in the presence of
endpoint singularities \cite{zaky2025high}.

Motivated by these challenges, the present work introduces a class of
fractional backward orthogonal functions designed for the approximation of
solutions with limited regularity near the terminal point. The proposed basis
is generated through a backward endpoint mapping that incorporates the terminal
singular structure directly into the approximation space. In contrast to usual
polynomial bases, the fractional backward Jacobi framework captures weak
regularity at the terminal endpoint and therefore provides an effective
approximation tool for weakly singular adjoint Volterra integral equations. More precisely, this paper develops and analyzes a fractional backward
spectral-collocation method for the weakly singular adjoint Volterra integral
equation
\begin{equation}\label{eq:3.d1R}
u(t)
=
g(t)+\int_t^1(\varrho-t)^{-\theta}K(t,\varrho)u(\varrho)\,d\varrho,
\qquad t\in J:=[0,1],
\end{equation}
where \(0<\theta<1\), \(g\in C(J)\), \(K\in C(J\times J)\), \(K(t,t)\ne 0\) for \(t\in J\), and
\(\theta\) denotes the weak-singularity exponent.

The main objective of this paper is to develop a fractional backward spectral
approximation theory for weakly singular adjoint integral equations. To this
end, we first establish the fundamental structural properties of the fractional
backward orthogonal functions, including orthogonality relations, recurrence
formulae, derivative identities, and the associated singular
Sturm--Liouville formulation. We then introduce suitable weighted Sobolev
spaces generated by the backward transformed derivative and derive projection
and Gauss-type interpolation error estimates. In addition, inverse
inequalities, weighted Sobolev estimates, and stability bounds for the weakly
singular adjoint Volterra operator are proved. These results provide the
analytical foundation for the fractional backward spectral collocation scheme
developed later in the paper.

The remainder of the paper is organized as follows.
Section~\ref{sec2} establishes the analytical properties of the fractional
backward orthogonal functions, including orthogonality relations, recurrence
formulae, derivative identities, and the associated singular
Sturm--Liouville framework. Section~\ref{sec3} develops the weighted
orthogonal projection theory and derives the corresponding approximation
estimates in weighted Sobolev spaces generated by the backward transformed
derivative \(\mathcal{D}_{\rho}\). Section~\ref{sec4} studies the associated
Gauss-type interpolation framework and proves stability estimates, inverse
inequalities, weighted Sobolev inequalities, and bounds for the weakly
singular adjoint Volterra operator. Section~\ref{sec5} constructs the
fractional backward spectral collocation method for weakly singular adjoint
Volterra integral equations and establishes convergence estimates in both the
\(L^{\infty}(J)\)-norm and the weighted
\(L^{2}_{\varkappa^{\mu,\upsilon,\rho}}(J)\)-norm. Section~\ref{sec6}
presents numerical experiments that illustrate the accuracy and convergence
behaviour of the proposed method. Finally, Section~\ref{sec7} summarizes the
main conclusions.

\section{Backward orthogonal functions}
\label{sec2}

This section collects the main structural properties of the fractional backward
orthogonal functions used throughout the paper. These functions are obtained by
composing classical Jacobi polynomials with a backward endpoint algebraic
mapping, thereby allowing the approximation space to incorporate terminal
endpoint singular behaviour explicitly.

We first define the family of backward \(\rho\)-polynomial spaces by
\begin{equation}\label{eq:R-poly-space}
P_r^\rho(J):=\mathrm{span}\Big\{1,(1-t)^{\rho},(1-t)^{2\rho},\ldots,(1-t)^{r\rho}\Big\},
\qquad J:=[0,1],\quad 0<\rho\le 1.
\end{equation}
Accordingly, any backward \(\rho\)-polynomial of degree \(r\) can be written as
\begin{equation}\label{eq:R-poly}
p_r^\rho(t):=
k_r(1-t)^{r\rho}+k_{r-1}(1-t)^{(r-1)\rho}+\cdots+k_1(1-t)^\rho+k_0,
\qquad k_r\neq 0,
\end{equation}
where \(\{k_i\}_{i=0}^r\) are real coefficients. We denote its degree by
\(\deg(p_r^\rho)=r\).

Let \(\varkappa\in L^{1}(J)\) be a positive weight function, and let
\(\{p_r^\rho\}_{r=0}^{\infty}\) be a sequence of \(\rho\)-polynomials satisfying
\(\deg(p_r^\rho)=r\). The sequence is said to be orthogonal with respect to
\(\varkappa\) if
\begin{equation}\label{eq:R-general-orth}
(p_r^\rho,p_s^\rho)_\varkappa
=
\int_0^1 p_r^\rho(t)p_s^\rho(t)\varkappa(t)\,dt
=
\beta_r\delta_{s,r},
\end{equation}
where
\[
\beta_r=(p_r^\rho,p_r^\rho)_\varkappa=\|p_r^\rho\|_{0,\varkappa}^{2}.
\]
Equivalently, the same function space can be generated by this orthogonal
basis as
\begin{equation}\label{eq:R-poly-generated}
P_r^\rho(J):=\mathrm{span}\Big\{p_0^\rho,p_1^\rho,\ldots,p_r^\rho\Big\}.
\end{equation}

The following elementary properties follow by adapting the standard arguments
for classical orthogonal polynomials; see, for example,
\cite[pp.~48--50]{shen2011spectral}.

\begin{lemma}\label{lem:R-orth-basic}
Let \(q\) be any element of \(P_r^\rho(J)\). Then \(q\) is orthogonal to
\(p_{r+1}^\rho\) with respect to the weighted inner product
\((\cdot,\cdot)_\varkappa\); that is,
\[
\bigl(p_{r+1}^\rho,q\bigr)_\varkappa=0.
\]
Consequently, \(p_{r+1}^\rho\) is orthogonal to the entire space
\(P_r^\rho(J)\) in the weighted space induced by \(\varkappa\).
\end{lemma}

\begin{lemma}\label{lem:R-monic}
Let \(\varkappa\in L^1(J)\) be a positive weight function. Then there exists
one and only one monic orthogonal system
\(\{\bar p_r^\rho\}_{r=0}^{\infty}\) of \(\rho\)-polynomials such that
\(\deg(\bar p_r^\rho)=r\) for every \(r\ge 0\). This system is determined by
\begin{equation}\label{eq:R-recurrence-0}
\bar p_0^\rho=1,\qquad
\bar p_1^\rho=(1-t)^\rho-\mu_0,
\end{equation}
and satisfies the three-term recurrence relation
\begin{equation}\label{eq:R-recurrence}
\bar p_{r+1}^\rho
=
\bigl((1-t)^\rho-\mu_r\bigr)\bar p_r^\rho
-\upsilon_r\bar p_{r-1}^\rho,
\qquad r\ge 1,
\end{equation}
where the recurrence coefficients are given by
\begin{equation}\label{eq:R-alpha}
\mu_r
=
\frac{\bigl((1-t)^\rho\bar p_r^\rho,\bar p_r^\rho\bigr)_\varkappa}
{\|\bar p_r^\rho\|_{0,\varkappa}^2},
\qquad r\ge 0,
\end{equation}
and
\begin{equation}\label{eq:R-beta}
\upsilon_r
=
\frac{\|\bar p_r^\rho\|_{0,\varkappa}^2}
{\|\bar p_{r-1}^\rho\|_{0,\varkappa}^2},
\qquad r\ge 1.
\end{equation}
\end{lemma}

We now introduce the fractional backward orthogonal functions used in this
paper.

\begin{definition}\label{def:RJacobi}
Let \(\mu,\upsilon>-1\) and \(0<\rho\le 1\). The fractional backward
orthogonal function of degree \(r\) is defined by
\begin{equation}\label{eq:RJacobi-def}
\mathcal{P}_r^{\mu,\upsilon,\rho}(t)
=
\mathcal{P}_r^{\mu,\upsilon}\!\left(1-2(1-t)^\rho\right),
\qquad t\in J,
\end{equation}
where \(\mathcal{P}_r^{\mu,\upsilon}\) denotes the classical Jacobi polynomial
of degree \(r\) on \([-1,1]\).
\end{definition}

For \(\rho=1\), the above definition reduces to the standard shifted Jacobi
polynomial,
\[
\mathcal{P}_r^{\mu,\upsilon,1}(t)=\mathcal{P}_r^{\mu,\upsilon}(2t-1),
\]
whose corresponding shifted Jacobi weight on \(J\) is
\[
t^\upsilon(1-t)^\mu.
\]

Using the standard expansion formula for shifted Jacobi polynomials, we obtain
\begin{equation}\label{eq:RJacobi-series}
\mathcal{P}_r^{\mu,\upsilon,\rho}(t)
=
\frac{\Gamma(1+r+\mu)}{r!\Gamma(1+r+\mu+\upsilon)}
\sum_{k=0}^{r}
\binom{r}{k}
\frac{(-1)^k \Gamma(1+r+k+\mu+\upsilon)}{\Gamma(1+k+\mu)}
(1-t)^{\rho k}.
\end{equation}

To transfer the differential and orthogonality properties of the shifted basis
to the backward endpoint setting, we introduce the following transformed
first-order derivative.

\begin{definition}\label{def:R-derivative}
The backward transformed derivative is defined by
\begin{equation}\label{eq:R-derivative-def}
\mathcal{D}_\rho f(t)
:=
\frac{d}{d\bigl(1-(1-t)^\rho\bigr)}f(t)
=
\frac{(1-t)^{1-\rho}}{\rho}\,f'(t),
\qquad t\in J.
\end{equation}
For any integer \(k\ge 1\), its higher-order powers are defined recursively by
\begin{equation}\label{eq:R-derivative-k}
\mathcal{D}_\rho^k f(t)
=
\underbrace{\mathcal{D}_\rho\circ \mathcal{D}_\rho\circ\cdots\circ \mathcal{D}_\rho}_{k\text{ times}}f(t).
\end{equation}
\end{definition}

The associated one-sided limits are given by
\begin{equation}\label{eq:R-derivative-limit-plus}
{}^{+}\mathcal{D}_\rho f(t)
:=
\lim_{\Delta t\to 0^+}
\frac{f(t+\Delta t)-f(t)}
{\bigl(1-(1-t-\Delta t)^\rho\bigr)-\bigl(1-(1-t)^\rho\bigr)},
\end{equation}
and
\begin{equation}\label{eq:R-derivative-limit-minus}
{}^{-}\mathcal{D}_\rho f(t)
:=
\lim_{\Delta t\to 0^-}
\frac{f(t+\Delta t)-f(t)}
{\bigl(1-(1-t-\Delta t)^\rho\bigr)-\bigl(1-(1-t)^\rho\bigr)}.
\end{equation}
Thus, \(\mathcal{D}_\rho f(t)\) exists whenever
\[
{}^{+}\mathcal{D}_\rho f(t)={}^{-}\mathcal{D}_\rho f(t),
\]
in which case
\[
\mathcal{D}_\rho f(t)={}^{+}\mathcal{D}_\rho f(t)={}^{-}\mathcal{D}_\rho f(t).
\]
In particular, for \(\rho=1\), the operator \(\mathcal{D}_\rho\) coincides with
the ordinary derivative.

We also define the two associated weights
\begin{equation}\label{eq:Rweight}
\varkappa^{\mu,\upsilon,\rho}(t)
:=
\rho (1-t)^{\rho(\mu+1)-1}\bigl(1-(1-t)^\rho\bigr)^\upsilon,
\end{equation}
and
\begin{equation}\label{eq:Rweight-hat}
\widehat{\varkappa}^{\mu,\upsilon,\rho}(t)
:=
(1-t)^{\rho\mu}\bigl(1-(1-t)^\rho\bigr)^\upsilon
=
\rho^{-1}(1-t)^{1-\rho}\varkappa^{\mu,\upsilon,\rho}(t).
\end{equation}

\begin{lemma}\label{lem:R-orth}
The fractional backward orthogonal functions
\(\{\mathcal{P}_r^{\mu,\upsilon,\rho}\}_{r=0}^{\infty}\) form an orthogonal
system with respect to the weight \(\varkappa^{\mu,\upsilon,\rho}\), namely,
\begin{equation}\label{eq:R-orth}
\int_0^1
\varkappa^{\mu,\upsilon,\rho}(t)
\mathcal{P}_r^{\mu,\upsilon,\rho}(t)
\mathcal{P}_s^{\mu,\upsilon,\rho}(t)\,dt
=
\widehat{\beta}_r^{\mu,\upsilon}\delta_{s,r},
\end{equation}
where
\begin{equation}\label{eq:2.5}
\widehat{\beta}_r^{\mu,\upsilon}
=
\frac{\Gamma(r+\mu+1)\Gamma(r+\upsilon+1)}
{r!(2r+\mu+\upsilon+1)\Gamma(r+\mu+\upsilon+1)}.
\end{equation}
\end{lemma}

The following result establishes that the fractional backward orthogonal
functions arise as eigenfunctions of a singular Sturm--Liouville operator
associated with the backward endpoint mapping
\(t\mapsto 1-(1-t)^\rho\).

Define
\begin{equation}\label{eq:2.7}
\mathcal{L}_{\rho}^{\mu,\upsilon}f(t)
=
-\bigl(\widehat{\varkappa}^{\mu,\upsilon,\rho}(t)\bigr)^{-1}
\mathcal{D}_\rho
\left\{
(1-t)^{\rho(\mu+1)}
\bigl(1-(1-t)^\rho\bigr)^{\upsilon+1}
\mathcal{D}_\rho f(t)
\right\}.
\end{equation}

\begin{lemma}\label{lem:R-SL}
The sequence \(\{\mathcal{P}_r^{\mu,\upsilon,\rho}\}_{r=0}^\infty\) satisfies
the singular Sturm--Liouville eigenvalue problem
\begin{equation}\label{eq:R-SL}
\mathcal{L}_{\rho}^{\mu,\upsilon}\mathcal{P}_r^{\mu,\upsilon,\rho}(t)
=
\sigma_r^{\mu,\upsilon}\mathcal{P}_r^{\mu,\upsilon,\rho}(t),
\end{equation}
where
\[
\sigma_r^{\mu,\upsilon}=r(r+\mu+\upsilon+1).
\]
Equivalently, in terms of the ordinary derivative with respect to \(t\),
\begin{equation}\label{eq:R-SL-xform}
-\bigl(\varkappa^{\mu,\upsilon,\rho}(t)\bigr)^{-1}
\frac{d}{dt}
\left\{
\rho^{-1}(1-t)^{\rho\mu+1}
\bigl(1-(1-t)^\rho\bigr)^{\upsilon+1}
\frac{d}{dt}\mathcal{P}_r^{\mu,\upsilon,\rho}(t)
\right\}
=
\sigma_r^{\mu,\upsilon}\mathcal{P}_r^{\mu,\upsilon,\rho}(t).
\end{equation}
\end{lemma}

\begin{lemma}\label{lem:R-deriv}
Let \(k\) be an integer satisfying \(0\le k\le r\). Then the transformed
derivatives of the fractional backward orthogonal functions satisfy
\begin{equation}\label{eq:R-deriv-relation}
\mathcal{D}_\rho^k\mathcal{P}_r^{\mu,\upsilon,\rho}(t)
=
\widehat{d}_{r,k}^{\mu,\upsilon}
\mathcal{P}_{r-k}^{\mu+k,\upsilon+k,\rho}(t),
\end{equation}
where
\begin{equation}\label{eq:R-dnk}
\widehat{d}_{r,k}^{\mu,\upsilon}
=
\frac{\Gamma(r+k+\mu+\upsilon+1)}
{\Gamma(r+\mu+\upsilon+1)}.
\end{equation}
Moreover, these derivatives are orthogonal with respect to the shifted weight
\(\varkappa^{\mu+k,\upsilon+k,\rho}\), that is,
\begin{equation}\label{eq:R-deriv-orth}
\int_0^1
\varkappa^{\mu+k,\upsilon+k,\rho}(t)
\mathcal{D}_\rho^k\mathcal{P}_r^{\mu,\upsilon,\rho}(t)
\mathcal{D}_\rho^k\mathcal{P}_s^{\mu,\upsilon,\rho}(t)\,dt
=
\widehat{h}_{r,k}^{\mu,\upsilon}\delta_{s,r},
\end{equation}
with
\begin{equation}\label{eq:R-hnk}
\widehat{h}_{r,k}^{\mu,\upsilon}
=
\frac{
\Gamma(r+\mu+1)\Gamma(r+\upsilon+1)\Gamma(r+k+\mu+\upsilon+1)
}{
(2r+\mu+\upsilon+1)(r-k)!\Gamma^2(r+\mu+\upsilon+1)
}.
\end{equation}
\end{lemma}

\begin{proof}
We begin with the case \(k=1\). Using integration by parts, together with
Lemma~\ref{lem:R-SL} and the orthogonality property of
\(\{\mathcal{P}_r^{\mu,\upsilon,\rho}\}_{r=0}^{\infty}\), we obtain
\begin{equation}\label{eq:R-proof-1}
\begin{aligned}
\int_{0}^{1}
\varkappa^{\mu+1,\upsilon+1,\rho}(t)
\mathcal{D}_{\rho}\mathcal{P}_r^{\mu,\upsilon,\rho}(t)
\mathcal{D}_{\rho}\mathcal{P}_s^{\mu,\upsilon,\rho}(t)\,dt
&=
\bigl(
\mathcal{P}_r^{\mu,\upsilon,\rho},
\mathcal{L}_{\rho}^{\mu,\upsilon}
\mathcal{P}_s^{\mu,\upsilon,\rho}
\bigr)_{\varkappa^{\mu,\upsilon,\rho}} \\
&=
\sigma_r^{\mu,\upsilon}\widehat{\beta}_r^{\mu,\upsilon}\delta_{s,r}.
\end{aligned}
\end{equation}
Hence, the family
\(\{\mathcal{D}_{\rho}\mathcal{P}_r^{\mu,\upsilon,\rho}\}_{r=1}^{\infty}\) is
orthogonal with respect to the weight
\(\varkappa^{\mu+1,\upsilon+1,\rho}\).

Furthermore, since
\(\mathcal{D}_{\rho}\mathcal{P}_r^{\mu,\upsilon,\rho}\in P_{r-1}^{\rho}(J)\),
Lemma~\ref{lem:R-monic} implies that it must be proportional to
\(\mathcal{P}_{r-1}^{\mu+1,\upsilon+1,\rho}\). Thus,
\begin{equation}\label{eq:R-proof-2}
\mathcal{D}_{\rho}\mathcal{P}_r^{\mu,\upsilon,\rho}(t)
=
\widehat{d}_{r,1}^{\mu,\upsilon}
\mathcal{P}_{r-1}^{\mu+1,\upsilon+1,\rho}(t),
\end{equation}
where \(\widehat{d}_{r,1}^{\mu,\upsilon}\) is independent of \(t\).

From \eqref{eq:RJacobi-series}, the leading coefficient of
\(\mathcal{P}_r^{\mu,\upsilon,\rho}\), regarded as a polynomial in
\(z:=(1-t)^\rho\), is
\[
k_r^{\mu,\upsilon}
=
(-1)^r\frac{\Gamma(2r+\mu+\upsilon+1)}
{r!\Gamma(r+\mu+\upsilon+1)}.
\]
Since
\[
\mathcal{D}_\rho
=
\frac{d}{d(1-(1-t)^\rho)}
=
-\frac{d}{dz},
\]
comparison of the leading coefficients in \eqref{eq:R-proof-2} yields
\begin{equation}\label{eq:R-proof-3}
\widehat{d}_{r,1}^{\mu,\upsilon}
=
\frac{-r\,k_r^{\mu,\upsilon}}{k_{r-1}^{\mu+1,\upsilon+1}}
=
r+\mu+\upsilon+1.
\end{equation}
Consequently,
\begin{equation}\label{eq:R-proof-4}
\mathcal{D}_{\rho}\mathcal{P}_r^{\mu,\upsilon,\rho}(t)
=
(r+\mu+\upsilon+1)
\mathcal{P}_{r-1}^{\mu+1,\upsilon+1,\rho}(t).
\end{equation}

Applying \eqref{eq:R-proof-4} recursively gives
\begin{equation}\label{eq:R-proof-5}
\mathcal{D}_{\rho}^{k}\mathcal{P}_r^{\mu,\upsilon,\rho}(t)
=
\widehat{d}_{r,k}^{\mu,\upsilon}
\mathcal{P}_{r-k}^{\mu+k,\upsilon+k,\rho}(t),
\end{equation}
where
\begin{equation}\label{eq:R-proof-6}
\widehat{d}_{r,k}^{\mu,\upsilon}
=
\frac{\Gamma(r+k+\mu+\upsilon+1)}
{\Gamma(r+\mu+\upsilon+1)}.
\end{equation}
This proves \eqref{eq:R-deriv-relation}.

Finally, applying Lemma~\ref{lem:R-orth} with the shifted parameters
\((\mu+k,\upsilon+k)\) gives
\begin{equation}\label{eq:R-proof-7}
\int_0^1
\varkappa^{\mu+k,\upsilon+k,\rho}(t)
\mathcal{D}_{\rho}^{k}\mathcal{P}_r^{\mu,\upsilon,\rho}(t)
\mathcal{D}_{\rho}^{k}\mathcal{P}_s^{\mu,\upsilon,\rho}(t)\,dt
=
\widehat{h}_{r,k}^{\mu,\upsilon}\delta_{s,r},
\end{equation}
where
\begin{equation}\label{eq:R-proof-8}
\begin{aligned}
\widehat{h}_{r,k}^{\mu,\upsilon}
&=
\bigl(\widehat{d}_{r,k}^{\mu,\upsilon}\bigr)^2
\widehat{\beta}_{r-k}^{\mu+k,\upsilon+k} \\
&=
\frac{
\Gamma(r+\mu+1)\Gamma(r+\upsilon+1)\Gamma(r+k+\mu+\upsilon+1)
}{
(2r+\mu+\upsilon+1)(r-k)!\Gamma^2(r+\mu+\upsilon+1)
},
\qquad r\ge k.
\end{aligned}
\end{equation}
The proof is complete.
\end{proof}

\begin{lemma}
\label{lem:right-gronwall}
Let \(u\in C(J)\) be nonnegative on \(J:=[0,1]\), and let
\(\varphi\in C(J)\) be nonnegative. Assume that \(0<\theta<1\), and let
\(K\in C(\mathcal{D}_R)\), where
\[
\mathcal{D}_R:=\{(t,\varrho):0\le t\le \varrho\le 1\}.
\]
Define
\[
(\mathcal{K}_R u)(t)
:=
\int_t^1
\frac{K(t,\varrho)}{(\varrho-t)^\theta}u(\varrho)\,d\varrho.
\]
If
\begin{equation}\label{eq:right-gronwall-assumption}
u(t)
\le
\varphi(t)+(\mathcal{K}_R u)(t),
\qquad 0\le t\le 1,
\end{equation}
then there exists a positive constant \(C\), independent of \(u\) and \(t\),
such that
\begin{equation}\label{eq:right-gronwall-estimate}
u(t)
\le
\Phi(t)
E_{1-\theta}
\!\left(
C\Gamma(1-\theta)(1-t)^{1-\theta}
\right),
\qquad 0\le t\le 1,
\end{equation}
where
\[
\Phi(t):=\max_{t\le s\le 1}\varphi(s),
\]
and
\[
E_{\sigma}(z)
=
\sum_{n=0}^{\infty}
\frac{z^n}{\Gamma(\sigma n+1)},
\qquad \sigma>0,
\]
is the Mittag--Leffler function.
\end{lemma}

\section{Orthogonal projection}
\label{sec3}

This section develops the orthogonal projection framework associated with the
fractional backward orthogonal functions introduced in Section~\ref{sec2}. The
main purpose is to establish weighted projection estimates in Sobolev spaces
generated by the backward transformed derivative \(\mathcal{D}_\rho\). These
estimates provide the approximation-theoretic foundation for the fractional
backward spectral approximation of functions with terminal-endpoint weak
singularities.

Let
\[
\Pi_{N,\varkappa^{\mu,\upsilon,\rho}}:
L^{2}_{\varkappa^{\mu,\upsilon,\rho}}(J)\to P^{\rho}_{N}(J)
\]
denote the \(L^{2}_{\varkappa^{\mu,\upsilon,\rho}}(J)\)-orthogonal projection
onto \(P^{\rho}_{N}(J)\). For each
\(f\in L^{2}_{\varkappa^{\mu,\upsilon,\rho}}(J)\), the projected function
\(\Pi_{N,\varkappa^{\mu,\upsilon,\rho}}f\in P^{\rho}_{N}(J)\) is characterized by
\begin{equation}\label{eq:2.13R-a}
\bigl(f-\Pi_{N,\varkappa^{\mu,\upsilon,\rho}}f,f_N\bigr)_{\varkappa^{\mu,\upsilon,\rho}}=0,
\qquad \forall f_N\in P^{\rho}_{N}(J).
\end{equation}
Equivalently, it admits the expansion
\begin{equation}\label{eq:2.13R}
\Pi_{N,\varkappa^{\mu,\upsilon,\rho}}f(t)
=
\sum_{r=0}^{N}\hat f_r^{\mu,\upsilon}
\mathcal{P}_r^{\mu,\upsilon,\rho}(t),
\end{equation}
where
\begin{equation}\label{eq:2.13R-b}
\hat f_r^{\mu,\upsilon}
=
\frac{\bigl(f,\mathcal{P}_r^{\mu,\upsilon,\rho}\bigr)_{\varkappa^{\mu,\upsilon,\rho}}}
{\|\mathcal{P}_r^{\mu,\upsilon,\rho}\|^2_{0,\varkappa^{\mu,\upsilon,\rho}}}.
\end{equation}
The best-approximation property of the orthogonal projection gives
\begin{equation}\label{eq:2.13R-c}
\|f-\Pi_{N,\varkappa^{\mu,\upsilon,\rho}}f\|_{0,\varkappa^{\mu,\upsilon,\rho}}
=
\inf_{f_N\in P_N^\rho(J)}
\|f-f_N\|_{0,\varkappa^{\mu,\upsilon,\rho}}.
\end{equation}

To estimate the projection error, we introduce the Jacobi-weighted Sobolev
space associated with the backward endpoint transformation:
\begin{equation}\label{eq:2.13R-d}
B_{\mu,\upsilon}^{s,\rho}(J)
:=
\bigl\{
f:\ \mathcal{D}_\rho^k f\in L^2_{\varkappa^{\mu+k,\upsilon+k,\rho}}(J),
\ 0\le k\le s
\bigr\},
\qquad s\in\mathbb{N}.
\end{equation}
This space is equipped with the inner product, norm, and seminorm
\begin{align}
(g,f)_{B_{\mu,\upsilon}^{s,\rho}}
&=
\sum_{k=0}^{s}
\bigl(\mathcal{D}_\rho^k g,\mathcal{D}_\rho^k f\bigr)_{\varkappa^{\mu+k,\upsilon+k,\rho}},
\label{eq:2.13R-e}\\
\|f\|_{B_{\mu,\upsilon}^{s,\rho}}
&=
(f,f)_{B_{\mu,\upsilon}^{s,\rho}}^{1/2},
\label{eq:2.13R-f}\\
|f|_{B_{\mu,\upsilon}^{s,\rho}}
&=
\|\mathcal{D}_\rho^s f\|_{0,\varkappa^{\mu+s,\upsilon+s,\rho}}.
\label{eq:2.13R-g}
\end{align}

When \(\rho=1\), the above space reduces to the classical shifted
Jacobi-weighted Sobolev space
\begin{equation}\label{eq:2.13R-h}
B_{\mu,\upsilon}^{s,1}(J)
:=
\bigl\{
f:\ \partial_t^k f\in L^2_{\varkappa^{\mu+k,\upsilon+k,1}}(J),
\ 0\le k\le s
\bigr\},
\qquad s\in\mathbb{N}.
\end{equation}
The following lemma connects the fractional backward Sobolev scale with the
corresponding classical shifted Jacobi scale.

\begin{lemma}\label{lem:2.6R}
A function \(f\) belongs to \(B_{\mu,\upsilon}^{s,\rho}(J)\) if and only if
\[
f\bigl(1-(1-t)^{1/\rho}\bigr)\in B_{\mu,\upsilon}^{s,1}(J).
\]
\end{lemma}

\begin{proof}
Set
\[
\eta=1-(1-t)^\rho,
\qquad
t=1-(1-\eta)^{1/\rho}.
\]
By the definition of \(\mathcal{D}_\rho\), we have
\[
\mathcal{D}_\rho^k f(t)
=
\partial_\eta^k\Bigl\{f\bigl(1-(1-\eta)^{1/\rho}\bigr)\Bigr\}.
\]
Therefore, for \(0\le k\le s\),
\begin{align}
\|\mathcal{D}_\rho^k f\|^2_{0,\varkappa^{\mu+k,\upsilon+k,\rho}}
&=
\int_0^1
\varkappa^{\mu+k,\upsilon+k,\rho}(t)
\bigl(\mathcal{D}_\rho^k f(t)\bigr)^2\,dt
\nonumber\\
&=
\int_0^1
\rho (1-t)^{\rho(\mu+k+1)-1}
\bigl(1-(1-t)^\rho\bigr)^{\upsilon+k}
\bigl(\mathcal{D}_\rho^k f(t)\bigr)^2\,dt
\nonumber\\
&=
\int_0^1
(1-\eta)^{\mu+k}\eta^{\upsilon+k}
\Bigl(\partial_\eta^k
\bigl\{f(1-(1-\eta)^{1/\rho})\bigr\}\Bigr)^2\,d\eta
\nonumber\\
&=
\Bigl\|
\partial_t^k\bigl\{f(1-(1-t)^{1/\rho})\bigr\}
\Bigr\|^2_{0,\varkappa^{\mu+k,\upsilon+k,1}}.
\label{eq:2.14R}
\end{align}
Summing over \(k=0,\ldots,s\) proves the equivalence of the two weighted
Sobolev regularities.
\end{proof}

\begin{proposition}\label{prop:2.1R}
Let \(0\le \jmath\le s\le N+1\). If
\[
f\bigl(1-(1-t)^{1/\rho}\bigr)\in B_{\mu,\upsilon}^{s,1}(J),
\]
then the orthogonal projector \(\Pi_{N,\varkappa^{\mu,\upsilon,\rho}}\)
satisfies
\begin{equation}\label{eq:2.15R}
\|\mathcal{D}_\rho^\jmath(f-\Pi_{N,\varkappa^{\mu,\upsilon,\rho}}f)\|_{0,\varkappa^{\mu+\jmath,\upsilon+\jmath,\rho}}
\le
c\sqrt{\frac{(N-s+1)!}{(N-\jmath+1)!}}\,N^{(\jmath-s)/2}
\Bigl\|
\partial_t^s\bigl\{f(1-(1-t)^{1/\rho})\bigr\}
\Bigr\|_{0,\varkappa^{\mu+s,\upsilon+s,1}}.
\end{equation}
For fixed \(s\), this estimate simplifies to
\begin{equation}\label{eq:2.16R}
\|\mathcal{D}_\rho^\jmath(f-\Pi_{N,\varkappa^{\mu,\upsilon,\rho}}f)\|_{0,\varkappa^{\mu+\jmath,\upsilon+\jmath,\rho}}
\le
cN^{\,\jmath-s}
\Bigl\|
\partial_t^s\bigl\{f(1-(1-t)^{1/\rho})\bigr\}
\Bigr\|_{0,\varkappa^{\mu+s,\upsilon+s,1}},
\end{equation}
where \(c\) is independent of \(N\) and \(f\). In particular,
\begin{equation}\label{eq:2.17R}
\|f-\Pi_{N,\varkappa^{\mu,\upsilon,\rho}}f\|_{0,\varkappa^{\mu,\upsilon,\rho}}
\le
cN^{-s}
\Bigl\|
\partial_t^s\bigl\{f(1-(1-t)^{1/\rho})\bigr\}
\Bigr\|_{0,\varkappa^{\mu+s,\upsilon+s,1}},
\end{equation}
and
\begin{equation}\label{eq:2.18R}
\|\partial_t(f-\Pi_{N,\varkappa^{\mu,\upsilon,\rho}}f)\|_{0,\tilde\varkappa^{\mu,\upsilon,\rho}}
\le
cN^{1-s}
\Bigl\|
\partial_t^s\bigl\{f(1-(1-t)^{1/\rho})\bigr\}
\Bigr\|_{0,\varkappa^{\mu+s,\upsilon+s,1}},
\end{equation}
where
\begin{equation}\label{eq:2.18R-a}
\tilde\varkappa^{\mu,\upsilon,\rho}(t)
=
\rho^{-1}(1-t)^{\rho\mu+1}
\bigl(1-(1-t)^\rho\bigr)^{\upsilon+1}.
\end{equation}
\end{proposition}

\begin{proof}
Assume that
\[
f\bigl(1-(1-t)^{1/\rho}\bigr)\in B_{\mu,\upsilon}^{s,1}(J).
\]
By Lemma~\ref{lem:2.6R}, this condition is equivalent to
\(f\in B_{\mu,\upsilon}^{s,\rho}(J)\). Using the expansion
\eqref{eq:2.13R} and the derivative orthogonality relation of the fractional
backward orthogonal functions, we obtain
\begin{equation}\label{eq:2.19R-a}
\|\mathcal{D}_\rho^\jmath(f-\Pi_{N,\varkappa^{\mu,\upsilon,\rho}}f)\|^2_{0,\varkappa^{\mu+\jmath,\upsilon+\jmath,\rho}}
=
\sum_{r=N+1}^{\infty}\widehat{h}_{r,\jmath}^{\mu,\upsilon}
|\hat f_r^{\mu,\upsilon}|^2
=
\sum_{r=N+1}^{\infty}
\frac{\widehat{h}_{r,\jmath}^{\mu,\upsilon}}
{\widehat{h}_{r,s}^{\mu,\upsilon}}
\widehat{h}_{r,s}^{\mu,\upsilon}
|\hat f_r^{\mu,\upsilon}|^2.
\end{equation}
Hence,
\begin{align}
\|\mathcal{D}_\rho^\jmath(f-\Pi_{N,\varkappa^{\mu,\upsilon,\rho}}f)\|^2_{0,\varkappa^{\mu+\jmath,\upsilon+\jmath,\rho}}
&\le
\max_{r\ge N+1}
\left\{
\frac{\widehat{h}_{r,\jmath}^{\mu,\upsilon}}
{\widehat{h}_{r,s}^{\mu,\upsilon}}
\right\}
\sum_{r=0}^{\infty}
\widehat{h}_{r,s}^{\mu,\upsilon}
|\hat f_r^{\mu,\upsilon}|^2
\nonumber\\
&\le
\frac{\widehat{h}_{N+1,\jmath}^{\mu,\upsilon}}
{\widehat{h}_{N+1,s}^{\mu,\upsilon}}
\|\mathcal{D}_\rho^s f\|^2_{0,\varkappa^{\mu+s,\upsilon+s,\rho}}.
\label{eq:2.19R}
\end{align}
Using the explicit expression for \(\widehat{h}_{r,k}^{\mu,\upsilon}\), we find
that, for \(0\le \jmath\le s\le N+1\),
\begin{align}
\frac{\widehat{h}_{N+1,\jmath}^{\mu,\upsilon}}
{\widehat{h}_{N+1,s}^{\mu,\upsilon}}
&=
\frac{\Gamma(N+\mu+\upsilon+\jmath+2)}
{\Gamma(N+\mu+\upsilon+s+2)}
\frac{(N-s+1)!}{(N-\jmath+1)!}
\nonumber\\
&=
\frac{1}
{(N+\mu+\upsilon+\jmath+2)\cdots
(N+\mu+\upsilon+s+1)}
\frac{(N-s+1)!}{(N-\jmath+1)!}
\nonumber\\
&\le
N^{\,\jmath-s}
\frac{(N-s+1)!}{(N-\jmath+1)!},
\label{eq:2.20R}
\end{align}
provided that \(\mu+\upsilon+2>0\). Combining
\eqref{eq:2.19R}, \eqref{eq:2.20R}, and \eqref{eq:2.14R} gives
\eqref{eq:2.15R}.

To obtain the simplified asymptotic estimate \eqref{eq:2.16R}, we use the
standard Gamma-function ratio bound
\begin{equation}\label{eq:2.21R-a}
\frac{\Gamma(r+c)}{\Gamma(r+d)}
\le
\alpha_r^{c,d}r^{c-d},
\qquad r+c>1,\quad r+d>1,
\end{equation}
where
\begin{equation}\label{eq:2.21R-b}
\alpha_r^{c,d}
=
\exp\Bigl(
\frac{c-d}{2(r+d-1)}
+\frac{1}{12(r+c-1)}
+\frac{(c-1)(d-1)}{r}
\Bigr).
\end{equation}
Consequently,
\begin{equation}\label{eq:2.21R}
\frac{(N-s+1)!}{(N-\jmath+1)!}
=
\frac{\Gamma(N-s+2)}{\Gamma(N-\jmath+2)}
\le
\alpha_N^{2-s,2-\jmath}N^{\,\jmath-s},
\end{equation}
where \(\alpha_N^{2-s,2-\jmath}\) remains bounded and approaches one for fixed
\(s\) as \(N\to\infty\). Therefore, \eqref{eq:2.16R} follows from
\eqref{eq:2.15R}.

The estimate \eqref{eq:2.17R} follows from \eqref{eq:2.16R} by taking
\(\jmath=0\). Moreover, since
\[
\mathcal{D}_\rho w(t)=\frac{(1-t)^{1-\rho}}{\rho}w'(t),
\]
we have
\[
\|\mathcal{D}_\rho w\|_{0,\varkappa^{\mu+1,\upsilon+1,\rho}}
=
\|\partial_t w\|_{0,\tilde\varkappa^{\mu,\upsilon,\rho}},
\]
with \(\tilde\varkappa^{\mu,\upsilon,\rho}\) defined in
\eqref{eq:2.18R-a}. Hence, \eqref{eq:2.18R} follows from
\eqref{eq:2.16R} with \(\jmath=1\). The proof is complete.
\end{proof}

\section{Gauss-type interpolation}
\label{sec4}

This section develops the Gauss-type interpolation framework associated with
the fractional backward orthogonal functions. The corresponding stability
bounds, inverse inequalities, interpolation-error estimates, and operator
bounds are established in weighted spaces adapted to the terminal-endpoint
structure. These results will be used later to control the discrete residuals
arising in the fractional backward spectral collocation scheme.

Let \(h_{j,\rho}^{\mu,\upsilon}(t)\) denote the backward Zaky--Lagrange basis
function defined by
\begin{equation}\label{eq:2.22R}
h_{j,\rho}^{\mu,\upsilon}(t)
=
\prod_{\substack{i=0 \\ i\neq j}}^{N}
\frac{z(t)-z(t_i)}{z(t_j)-z(t_i)}=
\prod_{\substack{i=0 \\ i\neq j}}^{N}
\frac{(1-t)^\rho-(1-t_i)^\rho}
{(1-t_j)^\rho-(1-t_i)^\rho},
\qquad 0\le j\le N,
\end{equation}
where
$z(t)=1-(1-t)^\rho,$
and \(\{t_i\}_{i=0}^{N}\) are the Zaky--Jacobi--Gauss nodes in \(J=(0,1)\),
namely, the zeros of
\(\mathcal{P}_{N+1}^{\mu,\upsilon,\rho}(t)\). Hence,
\begin{equation}\label{eq:2.22R-a}
h_{j,\rho}^{\mu,\upsilon}(t_i)=\delta_{ij},
\qquad 0\le i,j\le N.
\end{equation}

Define
\[
z_i=z(t_i)=1-(1-t_i)^\rho,
\qquad 0\le i\le N.
\]
Then \(\{z_i\}_{i=0}^{N}\) are precisely the shifted Jacobi--Gauss nodes,
that is, the zeros of \(\mathcal{P}_{N+1}^{\mu,\upsilon,1}(z)\). Consequently,
\begin{equation}\label{eq:2.23R}
h_{s,\rho}^{\mu,\upsilon}(t)
=
h_{s,1}^{\mu,\upsilon}(z)
:=
\prod_{\substack{i=0 \\ i\neq s}}^{N}
\frac{z-z_i}{z_s-z_i},
\qquad 0\le s\le N.
\end{equation}

We define the Zaky--Jacobi--Gauss interpolation operator
\(J_{N,\rho}^{\mu,\upsilon}\) by
\begin{equation}\label{eq:2.24R}
J_{N,\rho}^{\mu,\upsilon}f(t)
=
\sum_{j=0}^{N}f(t_j)\,h_{j,\rho}^{\mu,\upsilon}(t).
\end{equation}
Under the backward endpoint transformation \(z=1-(1-t)^\rho\), this operator
can be written as
\begin{equation}\label{eq:2.24R-a}
J_{N,\rho}^{\mu,\upsilon}f(t)
=
\sum_{j=0}^{N}
f\!\left(1-(1-z_j)^{1/\rho}\right)
h_{j,1}^{\mu,\upsilon}(z)
=
J_{N,1}^{\mu,\upsilon}
\Bigl(f\!\left(1-(1-z)^{1/\rho}\right)\Bigr).
\end{equation}

\begin{lemma}\label{lem:2.7R}
For \(\rho=1\), the interpolation operator \(J_{N,1}^{\mu,\upsilon}\) satisfies
\begin{equation}\label{eq:2.25R-a}
\|J_{N,1}^{\mu,\upsilon}f\|_{0,\varkappa^{\mu,\upsilon,1}}
\le
c\Big(
\|f\|_{0,\varkappa^{\mu,\upsilon,1}}
+
N^{-1}
\|\partial_t f\|_{0,\varkappa^{\mu+1,\upsilon+1,1}}
\Big),
\end{equation}
for all \(f\in B_{\mu,\upsilon}^{1,1}(J)\).
\end{lemma}

\begin{proposition}\label{prop:interp-stability-R}
If
\[
f\!\left(1-(1-t)^{1/\rho}\right)\in B_{\mu,\upsilon}^{1,1}(J),
\]
then the interpolation operator \(J_{N,\rho}^{\mu,\upsilon}\) satisfies
\begin{equation}\label{eq:2.25R}
\|J_{N,\rho}^{\mu,\upsilon}f\|_{0,\varkappa^{\mu,\upsilon,\rho}}
\le
c\Big(
\|f\|_{0,\varkappa^{\mu,\upsilon,\rho}}
+
N^{-1}
\|\mathcal{D}_\rho f\|_{0,\varkappa^{\mu+1,\upsilon+1,\rho}}
\Big),
\end{equation}
where \(c\) is independent of \(N\) and \(f\).
\end{proposition}

\begin{proof}
Using the change of variables
\[
z=1-(1-t)^\rho,
\]
in \eqref{eq:2.24R}, we obtain
\begin{align}
\|J_{N,\rho}^{\mu,\upsilon}f\|^2_{0,\varkappa^{\mu,\upsilon,\rho}}
&=
\int_0^1
\Bigg[
\sum_{i=0}^{N}
f(t_i)\,h_{i,1}^{\mu,\upsilon}(z(t))
\Bigg]^2
\rho(1-t)^{\rho(\mu+1)-1}z(t)^\upsilon\,dt
\nonumber\\
&=
\int_0^1
\Bigg[
\sum_{i=0}^{N}
f\!\left(1-(1-z_i)^{1/\rho}\right)
h_{i,1}^{\mu,\upsilon}(z)
\Bigg]^2
(1-z)^\mu z^\upsilon\,dz
\nonumber\\
&=
\Big\|
J_{N,1}^{\mu,\upsilon}
\Bigl(f\!\left(1-(1-z)^{1/\rho}\right)\Bigr)
\Big\|^2_{0,\varkappa^{\mu,\upsilon,1}}.
\label{eq:2.26R}
\end{align}
Applying Lemma~\ref{lem:2.7R} gives
\begin{align}
\|J_{N,\rho}^{\mu,\upsilon}f\|_{0,\varkappa^{\mu,\upsilon,\rho}}
&=
\Big\|
J_{N,1}^{\mu,\upsilon}
\Bigl(f\!\left(1-(1-t)^{1/\rho}\right)\Bigr)
\Big\|_{0,\varkappa^{\mu,\upsilon,1}}
\nonumber\\
&\le
c\Big(
\|f(1-(1-t)^{1/\rho})\|_{0,\varkappa^{\mu,\upsilon,1}}
+
N^{-1}
\|\partial_t f(1-(1-t)^{1/\rho})\|_{0,\varkappa^{\mu+1,\upsilon+1,1}}
\Big).
\label{eq:2.27R}
\end{align}
The desired estimate follows directly from Lemma~\ref{lem:2.6R}.
\end{proof}

\begin{proposition}\label{prop:2.3R}
For any \(\phi \in P^{\rho}_{N}(J)\), the following inverse estimates hold:
\begin{equation}\label{eq:2.28R}
\Vert \partial _{t}\phi \Vert _{0,\tilde{\varkappa }^{\mu ,\upsilon ,\rho }}
\le
\sqrt{\sigma ^{\mu ,\upsilon }_{N}}
\Vert \phi \Vert _{0,\varkappa ^{\mu ,\upsilon ,\rho }},
\end{equation}
and
\begin{equation}\label{eq:2.29R}
\Vert \mathcal{D}^{k}_{\rho }\phi \Vert _{0,\varkappa ^{\mu +k,\upsilon +k,\rho }}
\le c N^{k}\Vert \phi \Vert _{0,\varkappa ^{\mu ,\upsilon ,\rho }},
\qquad k\ge 1,
\end{equation}
where
\[
\sigma_N^{\mu,\upsilon}=N(N+\mu+\upsilon+1),
\]
and \(c\) is independent of \(N\) and \(\phi\), for fixed \(k\).
\end{proposition}

\begin{proof}
For any \(\phi \in P^{\rho}_{N}(J)\), we write
\begin{equation}\label{eq:2.30R}
\phi (t)=\sum _{i=0}^{N}\hat{\phi }_{i}^{\mu ,\upsilon }
\mathcal{P}^{\mu ,\upsilon ,\rho }_{i}(t),
\qquad
\hat{\phi }_{i}^{\mu ,\upsilon }
=
\frac{(\phi ,\mathcal{P}^{\mu ,\upsilon ,\rho }_{i})_{\varkappa ^{\mu ,\upsilon ,\rho }}}
{\widehat{\beta }_{i}^{\mu ,\upsilon }},
\end{equation}
where \(\widehat{\beta }_{i}^{\mu ,\upsilon }\) is given in \eqref{eq:2.5}.
By orthogonality,
\[
\Vert \phi \Vert ^{2}_{0,\varkappa ^{\mu ,\upsilon ,\rho }}
=
\sum _{i=0}^{N}
\widehat{\beta }_{i}^{\mu ,\upsilon }
|\hat{\phi }_{i}^{\mu ,\upsilon }|^{2}.
\]

Applying \(\mathcal{D}_{\rho}\) to \eqref{eq:2.30R} and using the
Sturm--Liouville relation, we obtain
\begin{align}
\Vert \mathcal{D}_{\rho }\phi \Vert ^{2}_{0,\varkappa ^{\mu +1,\upsilon +1,\rho }}
&=
\Vert \partial _{t}\phi \Vert ^{2}_{0,\tilde{\varkappa }^{\mu ,\upsilon ,\rho }}
=
\sum _{i=0}^{N}
\sigma ^{\mu ,\upsilon }_{i}
\widehat{\beta }_{i}^{\mu ,\upsilon }
|\hat{\phi }_{i}^{\mu ,\upsilon }|^{2}
\nonumber\\
&\le
\sigma ^{\mu ,\upsilon }_{N}
\sum _{i=0}^{N}
\widehat{\beta }_{i}^{\mu ,\upsilon }
|\hat{\phi }_{i}^{\mu ,\upsilon }|^{2}
=
\sigma ^{\mu ,\upsilon }_{N}
\Vert \phi \Vert _{0,\varkappa ^{\mu ,\upsilon ,\rho }}^{2}.
\label{eq:2.28R-proof}
\end{align}
This proves \eqref{eq:2.28R}.

Next, applying \(\mathcal{D}_{\rho}^{k}\) to \eqref{eq:2.30R} and using
\eqref{eq:R-deriv-orth}, we get
\begin{align}
\Vert \mathcal{D}^{k}_{\rho }\phi \Vert ^{2}_{0,\varkappa ^{\mu +k,\upsilon +k,\rho }}
&=
\sum _{r=k}^{N}
\widehat{h}^{\mu ,\upsilon }_{r,k}
|\hat{\phi }^{\mu ,\upsilon }_{r}|^{2}
\nonumber\\
&=
\sum _{r=k}^{N}
\frac{\Gamma (r+k+\mu +\upsilon +1)r!}
{\Gamma (r+\mu +\upsilon +1)(r-k)!}
\widehat{\beta }^{\mu ,\upsilon }_{r}
|\hat{\phi }^{\mu ,\upsilon }_{r}|^{2}
\nonumber\\
&\le
\frac{\Gamma (N+k+\mu +\upsilon +1)N!}
{\Gamma (N+\mu +\upsilon +1)(N-k)!}
\Vert \phi \Vert ^{2}_{0,\varkappa ^{\mu ,\upsilon ,\rho }}.
\label{eq:2.31R}
\end{align}
The standard Gamma-function estimate
\begin{equation}\label{eq:2.32R}
\frac{\Gamma (N+k+\mu +\upsilon +1)N!}
{\Gamma (N+\mu +\upsilon +1)(N-k)!}
\le
cN^{2k}
\end{equation}
then gives \eqref{eq:2.29R}.
\end{proof}

We next derive an error bound for the fractional backward approximation
generated by the Zaky--Jacobi--Gauss interpolation operator.

\begin{proposition}\label{prop:2.4R}
Assume that
\[
f\bigl(1-(1-t)^{1/\rho}\bigr)\in B_{\mu,\upsilon}^{s,1}(J),
\qquad s\ge 1.
\]
Then, for any \(0\le \jmath\le s\le N+1\),
\begin{equation}\label{eq:2.33R-a}
\Vert \mathcal{D}_{\rho }^{\jmath}\bigl(f-J_{N,\rho }^{\mu ,\upsilon }f\bigr)\Vert _{0,\varkappa ^{\mu +\jmath,\upsilon +\jmath,\rho }}
\le
c\sqrt{\frac{(N-s+1)!}{N!}}\,N^{\,\jmath-(s+1)/2}
\Vert \partial _{t}^{s}\bigl\{f\bigl(1-(1-t)^{1/\rho }\bigr)\bigr\}\Vert _{0,\varkappa ^{\mu +s,\upsilon +s,1}}.
\end{equation}
For fixed \(s\), this estimate reduces to
\begin{equation}\label{eq:2.33R}
\Vert \mathcal{D}_{\rho }^{\jmath}\bigl(f-J_{N,\rho }^{\mu ,\upsilon }f\bigr)\Vert _{0,\varkappa ^{\mu +\jmath,\upsilon +\jmath,\rho }}
\le
cN^{\jmath-s}
\Vert \partial _{t}^{s}\bigl\{f\bigl(1-(1-t)^{1/\rho }\bigr)\bigr\}\Vert _{0,\varkappa ^{\mu +s,\upsilon +s,1}},
\end{equation}
where \(c\) is independent of \(N\) and \(f\). In particular,
\begin{equation}\label{eq:2.34R}
\Vert f-J_{N,\rho }^{\mu ,\upsilon }f\Vert _{0,\varkappa ^{\mu ,\upsilon ,\rho }}
\le
cN^{-s}
\Vert \partial _{t}^{s}\bigl\{f\bigl(1-(1-t)^{1/\rho }\bigr)\bigr\}\Vert _{0,\varkappa ^{\mu +s,\upsilon +s,1}},
\end{equation}
and
\begin{equation}\label{eq:2.35R}
\Vert \partial _{t}(f-J_{N,\rho }^{\mu ,\upsilon }f)\Vert _{0,\tilde{\varkappa }^{\mu ,\upsilon ,\rho }}
\le
cN^{1-s}
\Vert \partial _{t}^{s}\bigl\{f\bigl(1-(1-t)^{1/\rho }\bigr)\bigr\}\Vert _{0,\varkappa ^{\mu +s,\upsilon +s,1}}.
\end{equation}
\end{proposition}

\begin{proof}
Combining the projection estimate \eqref{eq:2.15R} with
Proposition~\ref{prop:interp-stability-R}, we obtain
\begin{align}
\Vert J_{N,\rho }^{\mu ,\upsilon }f-\Pi _{N,\varkappa ^{\mu ,\upsilon ,\rho }}f\Vert _{0,\varkappa ^{\mu ,\upsilon ,\rho }}
&=
\Vert J_{N,\rho }^{\mu ,\upsilon }
\bigl(f-\Pi _{N,\varkappa ^{\mu ,\upsilon ,\rho }}f\bigr)
\Vert _{0,\varkappa ^{\mu ,\upsilon ,\rho }}
\nonumber\\
&\le
c\Bigl(
\Vert f-\Pi _{N,\varkappa ^{\mu ,\upsilon ,\rho }}f\Vert _{0,\varkappa ^{\mu ,\upsilon ,\rho }}
+
N^{-1}
\Vert \mathcal{D}_{\rho }
\bigl(f-\Pi _{N,\varkappa ^{\mu ,\upsilon ,\rho }}f\bigr)
\Vert _{0,\varkappa ^{\mu +1,\upsilon +1,\rho }}
\Bigr)
\nonumber\\
&\le
c\,N^{-(1+s)/2}
\sqrt{\frac{(N-s+1)!}{N!}}
\Vert \partial _{t}^{s}\bigl\{f\bigl(1-(1-t)^{1/\rho }\bigr)\bigr\}
\Vert _{0,\varkappa ^{\mu +s,\upsilon +s,1}}.
\label{eq:2.36R}
\end{align}
Using the inverse inequality \eqref{eq:2.29R}, we further have
\begin{align}
\Vert \mathcal{D}_{\rho }^{\jmath}
\bigl(J_{N,\rho }^{\mu ,\upsilon }f-\Pi _{N,\varkappa ^{\mu ,\upsilon ,\rho }}f\bigr)
\Vert _{0,\varkappa ^{\mu +\jmath,\upsilon +\jmath,\rho }}
&\le
cN^{\jmath}
\Vert J_{N,\rho }^{\mu ,\upsilon }f-\Pi _{N,\varkappa ^{\mu ,\upsilon ,\rho }}f
\Vert _{0,\varkappa ^{\mu ,\upsilon ,\rho }}
\nonumber\\
&\le
c\sqrt{\frac{(N-s+1)!}{N!}}\,N^{\,\jmath-(s+1)/2}
\Vert \partial _{t}^{s}\bigl\{f\bigl(1-(1-t)^{1/\rho }\bigr)\bigr\}
\Vert _{0,\varkappa ^{\mu +s,\upsilon +s,1}}.
\label{eq:2.36R-b}
\end{align}
By the triangle inequality and \eqref{eq:2.15R},
\begin{align}
\Vert \mathcal{D}_{\rho }^{\jmath}
\bigl(f-J_{N,\rho }^{\mu ,\upsilon }f\bigr)
\Vert _{0,\varkappa ^{\mu +\jmath,\upsilon +\jmath,\rho }}
&\le
\Vert \mathcal{D}_{\rho }^{\jmath}
\bigl(f-\Pi _{N,\varkappa ^{\mu ,\upsilon ,\rho }}f\bigr)
\Vert _{0,\varkappa ^{\mu +\jmath,\upsilon +\jmath,\rho }}
\nonumber\\
&\quad+
\Vert \mathcal{D}_{\rho }^{\jmath}
\bigl(\Pi _{N,\varkappa ^{\mu ,\upsilon ,\rho }}f-J_{N,\rho }^{\mu ,\upsilon }f\bigr)
\Vert _{0,\varkappa ^{\mu +\jmath,\upsilon +\jmath,\rho }}
\nonumber\\
&\le
c\sqrt{\frac{(N-s+1)!}{N!}}\,N^{\,\jmath-(s+1)/2}
\Vert \partial _{t}^{s}\bigl\{f\bigl(1-(1-t)^{1/\rho }\bigr)\bigr\}
\Vert _{0,\varkappa ^{\mu +s,\upsilon +s,1}}.
\end{align}
This proves \eqref{eq:2.33R-a}. The fixed-\(s\) estimate
\eqref{eq:2.33R} and the special cases \eqref{eq:2.34R}--\eqref{eq:2.35R}
follow by the same Gamma-ratio argument used in Proposition~\ref{prop:2.1R}.
\end{proof}

\begin{lemma}[Weighted Sobolev inequality]\label{lem:2.8R}
Let \(f\in B_{\mu,\upsilon}^{1,\rho}(J)\), and suppose that \(f(\xi)=0\) for
some point \(\xi\in[0,1]\). If
\[
-1<\mu,\upsilon\le -\frac12,
\]
then
\begin{equation}\label{eq:2.37R}
\|f\|_{\infty}
\le
\sqrt{2}\,\|f\|_{0,\varkappa^{\mu,\upsilon,\rho}}^{1/2}
\|\partial_t f\|_{0,\tilde{\varkappa}^{\mu,\upsilon,\rho}}^{1/2}.
\end{equation}
\end{lemma}

\begin{proof}
For any \(t\in[\xi,1]\), applying the fundamental theorem of calculus followed
by the Cauchy--Schwarz inequality gives
\begin{align}
f^{2}(t)
&=
\int_{\xi}^{t} d\bigl(f^{2}(z)\bigr)
=
\int_{\xi}^{t} 2f(z)\partial_{z}f(z)\,dz
\le
2\int_{\xi}^{1}|f(z)\partial_{z}f(z)|\,dz
\nonumber\\
&=
2\int_{\xi}^{1}
\rho^{1/2}(1-z)^{\frac{\rho(\mu+1)-1}{2}}
\bigl(1-(1-z)^{\rho}\bigr)^{\upsilon/2}|f(z)|
\nonumber\\
&\qquad\qquad\times
\Bigl[
\rho^{-1/2}(1-z)^{-\frac{\rho(\mu+1)-1}{2}}
\bigl(1-(1-z)^{\rho}\bigr)^{-\upsilon/2}
|\partial_{z}f(z)|
\Bigr]\,dz
\nonumber\\
&\le
2\|f\|_{0,\varkappa^{\mu,\upsilon,\rho}}
\Biggl[
\int_{0}^{1}
\rho^{-1}(1-z)^{-\rho(\mu+1)+1}
\bigl(1-(1-z)^{\rho}\bigr)^{-\upsilon}
\bigl(\partial_{z}f(z)\bigr)^{2}\,dz
\Biggr]^{1/2}.
\label{eq:2.38R}
\end{align}
Since \(-1<\mu,\upsilon\le -1/2\), the following bound holds for
all \(z\in J\):
\begin{equation}\label{eq:2.39R}
\rho^{-1}(1-z)^{-\rho(\mu+1)+1}
\bigl(1-(1-z)^{\rho}\bigr)^{-\upsilon}
\le
\rho^{-1}(1-z)^{\rho\mu+1}
\bigl(1-(1-z)^{\rho}\bigr)^{\upsilon+1}
=
\tilde{\varkappa}^{\mu,\upsilon,\rho}(z).
\end{equation}
Inserting \eqref{eq:2.39R} into \eqref{eq:2.38R} yields
\begin{equation}\label{eq:2.40R}
\max_{t\in[\xi,1]}|f(t)|
\le
\sqrt{2}\,\|f\|_{0,\varkappa^{\mu,\upsilon,\rho}}^{1/2}
\|\partial_t f\|_{0,\tilde{\varkappa}^{\mu,\upsilon,\rho}}^{1/2}.
\end{equation}
Repeating the same argument over the interval \([0,\xi]\), one obtains
\begin{equation}\label{eq:2.41R}
\max_{t\in[0,\xi]}|f(t)|
\le
\sqrt{2}\,\|f\|_{0,\varkappa^{\mu,\upsilon,\rho}}^{1/2}
\|\partial_t f\|_{0,\tilde{\varkappa}^{\mu,\upsilon,\rho}}^{1/2}.
\end{equation}
The desired estimate follows by combining \eqref{eq:2.40R} and
\eqref{eq:2.41R}.
\end{proof}

\begin{proposition}[Interpolation error in the \(L^\infty\)-norm]
\label{prop:2.5R}
Let \(-1<\mu,\upsilon\le -\dfrac{1}{2}\). If
\[
f\bigl(1-(1-t)^{1/\rho }\bigr)\in B^{s,1}_{\mu ,\upsilon }(J),
\qquad s\ge 1,
\]
then
\begin{equation}\label{eq:2.42R}
\Vert f-J_{N,\rho }^{\mu ,\upsilon }f\Vert _{\infty }
\le
c N^{1/2-s}
\Vert \partial ^{s}_{t}
\bigl\{f\bigl(1-(1-t)^{1/\rho }\bigr)\bigr\}
\Vert _{0,\varkappa ^{\mu +s,\upsilon +s,1}}.
\end{equation}
\end{proposition}

\begin{proof}
By Lemma~\ref{lem:2.6R},
\[
f\bigl(1-(1-t)^{1/\rho }\bigr)\in B^{s,1}_{\mu ,\upsilon }(J),
\qquad s\ge 1,
\]
implies \(f\in B_{\mu,\upsilon }^{1,\rho }(J)\). Moreover,
\(f-J_{N,\rho }^{\mu ,\upsilon }f\) vanishes at all interpolation nodes, i.e.,
at the roots of \(\mathcal{P}_{N+1}^{\mu ,\upsilon ,\rho }(t)\). Therefore,
Lemma~\ref{lem:2.8R} and Proposition~\ref{prop:2.4R} give
\begin{align}
\Vert f-J_{N,\rho }^{\mu ,\upsilon }f\Vert _{\infty }
&\le
\sqrt{2}
\Vert f-J_{N,\rho }^{\mu ,\upsilon }f
\Vert _{0,\varkappa ^{\mu ,\upsilon ,\rho }}^{1/2}
\Vert \partial _{t}\bigl(f-J_{N,\rho }^{\mu ,\upsilon }f\bigr)
\Vert _{0,\tilde{\varkappa }^{\mu ,\upsilon ,\rho }}^{1/2}
\nonumber\\
&\le
c N^{1/2-s}
\Vert \partial ^{s}_{t}
\bigl\{f\bigl(1-(1-t)^{1/\rho }\bigr)\bigr\}
\Vert _{0,\varkappa ^{\mu +s,\upsilon +s,1}}.
\end{align}
This proves the assertion.
\end{proof}

We next introduce the discrete inner product associated with the shifted
Jacobi--Gauss quadrature nodes. For \(g,f\in C(\bar{J})\), define
\begin{equation}\label{eq:2.43R-a}
(g,f)_{N,\varkappa ^{\mu ,\upsilon ,1}}
=
\sum _{i=0}^{N}{g}(z_{i})f(z_{i})\varkappa _{i},
\end{equation}
where \(\{z_i\}_{i=0}^{N}\) denote the zeros of
\(\mathcal{P}_{N+1}^{\mu,\upsilon,1}\), while
\(\{\varkappa_i\}_{i=0}^{N}\) are the associated quadrature weights.

\begin{lemma}[\cite{bernardi1997spectral,shen2011spectral}]
\label{lem:2.9R}
For every \(f\in B_{\mu,\upsilon}^{s,1}(J)\) with \(s\ge 1\), and every
\(\phi\in P_N^1(J)\), the quadrature error satisfies
\begin{equation}\label{eq:2.44R-a}
\bigl|(f,\phi )_{\varkappa ^{\mu ,\upsilon ,1}}
-(f,\phi )_{N,\varkappa ^{\mu ,\upsilon ,1}}\bigr|
\le
cN^{-s}
\|\partial_t^s f\|_{0,\varkappa^{\mu+s,\upsilon+s,1}}
\|\phi\|_{0,\varkappa^{\mu,\upsilon,1}}.
\end{equation}
\end{lemma}

We shall also use the following bound for the Lebesgue constant corresponding
to the Lagrange interpolation basis generated by the zeros of the shifted
Jacobi polynomials \cite{mastroianni2001optimal}.

\begin{lemma}\label{lem:2.10R}
Let \(\{h_{j,1}^{\mu,\upsilon}(t)\}_{j=0}^{N}\) be the Lagrange basis
associated with the shifted Jacobi--Gauss nodes determined by
\(\mathcal{P}_{N+1}^{\mu,\upsilon,1}\). Then
\begin{equation}\label{eq:2.45R-a}
\|J_{N,1}^{\mu,\upsilon}\|_{\infty}
:=
\max_{t\in J}\sum_{j=0}^{N}|h_{j,1}^{\mu,\upsilon}(t)|
=
\left\{
\begin{array}{ll}
O(\log N), & -1<\mu,\upsilon\le -\frac{1}{2},\\[1mm]
O(N^{\beta+\frac{1}{2}}), & \beta=\max(\mu,\upsilon),\ \text{otherwise}.
\end{array}
\right.
\end{equation}
\end{lemma}

An analogous estimate holds for the generalized Lagrange basis constructed
from the Zaky--Jacobi--Gauss nodes.

\begin{lemma}\label{lem:2.11R}
Let \(\{h_{j,\rho}^{\mu,\upsilon}(t)\}_{j=0}^{N}\) be the generalized Lagrange
basis associated with the zeros of
\(\mathcal{P}_{N+1}^{\mu,\upsilon,\rho}(t)\). Then
\begin{equation}\label{eq:2.46R-a}
\|J_{N,\rho}^{\mu,\upsilon}\|_{\infty}
:=
\max_{t\in J}\sum_{j=0}^{N}|h_{j,\rho}^{\mu,\upsilon}(t)|
=
\left\{
\begin{array}{ll}
O(\log N), & -1<\mu,\upsilon\le -\frac{1}{2},\\[1mm]
O(N^{\beta+\frac{1}{2}}), & \beta=\max(\mu,\upsilon),\ \text{otherwise}.
\end{array}
\right.
\end{equation}
\end{lemma}

For the analysis of the weakly singular adjoint integral terms below, we
employ the following H\"older spaces. For \(m\ge 0\) and \(\nu\in[0,1]\),
\(C^{m,\nu}(J)\) denotes the space of functions whose derivatives up to order
\(m\) are continuous on \(J\), with the \(m\)-th derivative being H\"older
continuous of exponent \(\nu\). The corresponding norm is defined by
\begin{equation}\label{eq:Holder-norm-R}
\|f\|_{m,\nu}
=
\max_{0\le i\le m}\max_{t\in J}
\left|\partial_t^i f(t)\right|
+
\max_{0\le i\le m}
\sup_{\substack{t,\varrho\in J\\ t\ne \varrho}}
\frac{
\left|\partial_t^i f(t)-\partial_t^i f(\varrho)\right|
}{
|t-\varrho|^\nu
}.
\end{equation}
When \(\nu=0\), the space \(C^{m,0}(J)\) coincides with the standard space
\(C^m(J)\), endowed with the norm \(\|\cdot\|_m\).

We shall rely on the following classical approximation result of Ragozin
\cite{ragozin1970polynomial,ragozin1971constructive}. For every nonnegative
integer \(m\) and every \(\nu\in(0,1)\), there exists a linear operator
\[
\mathscr{T}_N:C^{m,\nu}(J)\longrightarrow P_N^1(J)
\]
such that
\begin{equation}\label{eq:Ragozin-R}
\|f-\mathscr{T}_N f\|_{\infty}
\le
C_{m,\nu}N^{-(m+\nu)}\|f\|_{m,\nu},
\qquad f\in C^{m,\nu}(J),
\end{equation}
where \(C_{m,\nu}\) is independent of \(N\) and \(f\).

We next introduce the weakly singular adjoint Volterra integral operator
\(\mathcal{K}_R\), defined by
\begin{equation}\label{eq:KR-def}
(\mathcal{K}_R f)(t)
=
\int_t^1
(\varrho-t)^{-\theta}
K(t,\varrho)f(\varrho)\,d\varrho,
\qquad 0<\theta<1,\quad t\in J,
\end{equation}
where \(K\in C(J\times J)\) and \(K(t,t)\ne 0\) for all \(t\in J\).
We show below that \(\mathcal{K}_R\) maps \(C(J)\) continuously into
\(C^{0,\nu}(J)\) for every \(0<\nu<1-\theta\).

\begin{lemma}\label{lem:KR-holder}
Let \(f\in C(J)\), \(K\in C(J\times J)\), and assume that
\(K(\cdot,\varrho)\in C^{0,\nu}(J)\) with \(0<\nu<1-\theta\), then
\begin{equation}\label{eq:KR-holder-est}
\frac{|(\mathcal{K}_R f)(t)-(\mathcal{K}_R f)(s)|}{|t-s|^\nu}
\le
c\|f\|_{\infty},
\qquad
\forall\, t,s\in J,\quad t\ne s .
\end{equation}
Consequently,
\begin{equation}\label{eq:KR-holder-norm}
\|\mathcal{K}_R f\|_{0,\nu}
\le
c\|f\|_{\infty},
\qquad 0<\nu<1-\theta .
\end{equation}
\end{lemma}

\begin{proof}
It is enough to consider the case \(0\le s<t\le 1\). From the definition of
the weakly singular adjoint Volterra integral operator, we have
\[
(\mathcal{K}_R f)(t)
=
\int_t^1(\varrho-t)^{-\theta}K(t,\varrho)f(\varrho)\,d\varrho
\]
and
\[
(\mathcal{K}_R f)(s)
=
\int_s^1(\varrho-s)^{-\theta}K(s,\varrho)f(\varrho)\,d\varrho .
\]
Therefore,
\begin{align}
\frac{|(\mathcal{K}_R f)(t)-(\mathcal{K}_R f)(s)|}{|t-s|^\nu}
&=
(t-s)^{-\nu}
\left|
\int_t^1(\varrho-t)^{-\theta}K(t,\varrho)f(\varrho)\,d\varrho
-
\int_s^1(\varrho-s)^{-\theta}K(s,\varrho)f(\varrho)\,d\varrho
\right|
\nonumber\\
&\le
(t-s)^{-\nu}
\left|
\int_t^1
\Big[
(\varrho-t)^{-\theta}K(t,\varrho)
-
(\varrho-s)^{-\theta}K(s,\varrho)
\Big]f(\varrho)\,d\varrho
\right|
\nonumber\\
&\quad
+
(t-s)^{-\nu}
\left|
\int_s^t
(\varrho-s)^{-\theta}K(s,\varrho)f(\varrho)\,d\varrho
\right|
\nonumber\\
&\le B_1+B_2,
\label{eq:KR-B1B2}
\end{align}
where
\begin{align}
B_1
&=
(t-s)^{-\nu}
\int_t^1
\left|
(\varrho-t)^{-\theta}K(t,\varrho)
-
(\varrho-s)^{-\theta}K(s,\varrho)
\right|
|f(\varrho)|\,d\varrho,
\label{eq:KR-B1}\\
B_2
&=
(t-s)^{-\nu}
\int_s^t
(\varrho-s)^{-\theta}|K(s,\varrho)|\,|f(\varrho)|\,d\varrho .
\label{eq:KR-B2}
\end{align}

For \(B_1\), applying the triangle inequality gives
\begin{equation}\label{eq:KR-B1-split}
B_1\le B^{(1)}+B^{(2)},
\end{equation}
where
\begin{align}
B^{(1)}
&=
(t-s)^{-\nu}
\int_t^1
\left|
(\varrho-t)^{-\theta}
-
(\varrho-s)^{-\theta}
\right|
|K(t,\varrho)|\,|f(\varrho)|\,d\varrho,
\label{eq:KR-B11}\\
B^{(2)}
&=
(t-s)^{-\nu}
\int_t^1
(\varrho-s)^{-\theta}
|K(t,\varrho)-K(s,\varrho)|\,|f(\varrho)|\,d\varrho .
\label{eq:KR-B12}
\end{align}

Since \(K\in C(J\times J)\), there exists a constant \(c>0\) such that
\[
|K(t,\varrho)|\le c,
\qquad (t,\varrho)\in J\times J .
\]
Furthermore, for \(\varrho\ge t>s\), one has
\[
(\varrho-t)^{-\theta}-(\varrho-s)^{-\theta}\ge 0 .
\]
Thus,
\begin{align}
B^{(1)}
&\le
c\|f\|_{\infty}
(t-s)^{-\nu}
\int_t^1
\left[
(\varrho-t)^{-\theta}
-
(\varrho-s)^{-\theta}
\right]\,d\varrho
\nonumber\\
&=
c\|f\|_{\infty}
(t-s)^{-\nu}
\left[
\int_t^1(\varrho-t)^{-\theta}\,d\varrho
-
\int_t^1(\varrho-s)^{-\theta}\,d\varrho
\right]
\nonumber\\
&=
c\|f\|_{\infty}
(t-s)^{-\nu}
\left[
\int_t^1(\varrho-t)^{-\theta}\,d\varrho
-
\int_s^1(\varrho-s)^{-\theta}\,d\varrho
+
\int_s^t(\varrho-s)^{-\theta}\,d\varrho
\right]
\nonumber\\
&=
c\|f\|_{\infty}
(t-s)^{-\nu}
\left[
\frac{(1-t)^{1-\theta}}{1-\theta}
-
\frac{(1-s)^{1-\theta}}{1-\theta}
+
\frac{(t-s)^{1-\theta}}{1-\theta}
\right]
\nonumber\\
&\le
c\|f\|_{\infty}
(t-s)^{-\nu}
\frac{(t-s)^{1-\theta}}{1-\theta}
\nonumber\\
&=
c\|f\|_{\infty}
(t-s)^{1-\theta-\nu}
\nonumber\\
&\le
c\|f\|_{\infty},
\label{eq:KR-B11-est}
\end{align}
because \(0<t-s\le 1\) and \(0<\nu<1-\theta\). Equivalently, the above
integrals may be expressed in beta-function form as
\[
\int_s^t(\varrho-s)^{-\theta}\,d\varrho
=
(t-s)^{1-\theta}\int_0^1(1-\eta)^{-\theta}\,d\eta
=
B(1-\theta,1)(t-s)^{1-\theta}.
\]

For \(B^{(2)}\), the assumed uniform H\"older continuity of \(K(\cdot,\varrho)\)
with respect to the first variable gives
\[
|K(t,\varrho)-K(s,\varrho)|
\le
c|t-s|^\nu,
\qquad \varrho\in J.
\]
Hence,
\begin{align}
B^{(2)}
&\le
(t-s)^{-\nu}
\int_t^1
(\varrho-s)^{-\theta}
c|t-s|^\nu
|f(\varrho)|\,d\varrho
\nonumber\\
&\le
c\|f\|_{\infty}
\int_t^1(\varrho-s)^{-\theta}\,d\varrho
\nonumber\\
&\le
c\|f\|_{\infty}
\int_s^1(\varrho-s)^{-\theta}\,d\varrho
\nonumber\\
&=
c\|f\|_{\infty}
\frac{(1-s)^{1-\theta}}{1-\theta}
\nonumber\\
&\le
c\|f\|_{\infty}.
\label{eq:KR-B12-est}
\end{align}

For \(B_2\), boundedness of \(K\) yields
\begin{align}
B_2
&\le
c\|f\|_{\infty}
(t-s)^{-\nu}
\int_s^t(\varrho-s)^{-\theta}\,d\varrho
\nonumber\\
&=
c\|f\|_{\infty}
(t-s)^{-\nu}
\frac{(t-s)^{1-\theta}}{1-\theta}
\nonumber\\
&=
c\|f\|_{\infty}
(t-s)^{1-\theta-\nu}
\nonumber\\
&\le
c\|f\|_{\infty},
\label{eq:KR-B2-est}
\end{align}
again using \(0<\nu<1-\theta\).

Combining \eqref{eq:KR-B1B2}, \eqref{eq:KR-B1-split},
\eqref{eq:KR-B11-est}, \eqref{eq:KR-B12-est}, and
\eqref{eq:KR-B2-est}, we obtain
\[
\frac{|(\mathcal{K}_R f)(t)-(\mathcal{K}_R f)(s)|}{|t-s|^\nu}
\le
c\|f\|_{\infty},
\qquad t\ne s.
\]
This proves \eqref{eq:KR-holder-est}. Finally, since
\[
|(\mathcal{K}_R f)(t)|
\le
\|K\|_{\infty}\|f\|_{\infty}
\int_t^1(\varrho-t)^{-\theta}\,d\varrho
\le
c\|f\|_{\infty},
\]
we also have
\[
\|\mathcal{K}_R f\|_{\infty}\le c\|f\|_{\infty}.
\]
Together with \eqref{eq:KR-holder-est}, this implies
\[
\|\mathcal{K}_R f\|_{0,\nu}
\le
c\|f\|_{\infty}.
\]
The proof is complete.
\end{proof}

\begin{lemma}\label{lem:KR-holder-transform}
Let \(f\in C(J)\), \(K\in C(J\times J)\), and assume that
\(K(\cdot,\varrho)\in C^{0,\nu}(J)\) uniformly with respect to
\(\varrho\in J\). If \(0<\nu<1-\theta\), then
\begin{equation}\label{eq:KR-holder-transform}
\frac{
\left|
(\mathcal{K}_R f)\bigl(1-(1-t)^{1/\rho}\bigr)
-
(\mathcal{K}_R f)\bigl(1-(1-s)^{1/\rho}\bigr)
\right|
}
{|t-s|^\nu}
\le
c\|f\|_{\infty},
\qquad
\forall\, t,s\in J,\quad t\ne s .
\end{equation}
Consequently,
\begin{equation}\label{eq:KR-holder-transform-norm}
\left\|
(\mathcal{K}_R f)\bigl(1-(1-t)^{1/\rho}\bigr)
\right\|_{0,\nu}
\le
c\|f\|_{\infty}.
\end{equation}
\end{lemma}

\begin{proof}
By Lemma~\ref{lem:KR-holder},
\begin{equation}\label{eq:KR-holder-transform-proof-1}
\frac{
|(\mathcal{K}_R f)(\xi)-(\mathcal{K}_R f)(\eta)|
}
{|\xi-\eta|^\nu}
\le
c\|f\|_{\infty},
\qquad
\forall\, \xi,\eta\in J,\quad \xi\ne\eta .
\end{equation}
Set
\[
\xi=1-(1-t)^{1/\rho},
\qquad
\eta=1-(1-s)^{1/\rho}.
\]
Then
\begin{equation}\label{eq:KR-holder-transform-proof-2}
\frac{
\left|
(\mathcal{K}_R f)\bigl(1-(1-t)^{1/\rho}\bigr)
-
(\mathcal{K}_R f)\bigl(1-(1-s)^{1/\rho}\bigr)
\right|
}
{
\left|
\bigl(1-(1-t)^{1/\rho}\bigr)
-
\bigl(1-(1-s)^{1/\rho}\bigr)
\right|^\nu
}
\le
c\|f\|_{\infty}.
\end{equation}

It remains to compare the transformed distance with \(|t-s|\). Let
\[
\Phi(t)=1-(1-t)^{1/\rho},
\qquad 0<\rho\le 1.
\]
Then
\[
\Phi'(t)
=
\frac{1}{\rho}(1-t)^{1/\rho-1}.
\]
Since \(0<\rho\le 1\), it follows that
\[
0\le (1-t)^{1/\rho-1}\le 1,
\qquad t\in J.
\]
Therefore,
\begin{equation}\label{eq:Phi-Lipschitz}
|\Phi(t)-\Phi(s)|
\le
\frac{1}{\rho}|t-s|,
\qquad
\forall\, t,s\in J,
\end{equation}
and hence
\begin{equation}\label{eq:Phi-Lipschitz-kappa}
|\Phi(t)-\Phi(s)|^\nu
\le
\rho^{-\nu}|t-s|^\nu .
\end{equation}
Using \eqref{eq:KR-holder-transform-proof-2} and
\eqref{eq:Phi-Lipschitz-kappa}, we obtain
\begin{align}
&
\frac{
\left|
(\mathcal{K}_R f)\bigl(1-(1-t)^{1/\rho}\bigr)
-
(\mathcal{K}_R f)\bigl(1-(1-s)^{1/\rho}\bigr)
\right|
}
{|t-s|^\nu}
\nonumber\\
&\qquad\le
c\|f\|_{\infty}
\frac{|\Phi(t)-\Phi(s)|^\nu}{|t-s|^\nu}
\le
c\rho^{-\nu}\|f\|_{\infty}
\le
c\|f\|_{\infty}.
\end{align}
This proves \eqref{eq:KR-holder-transform}. The norm estimate
\eqref{eq:KR-holder-transform-norm} follows from the definition of
\(\|\cdot\|_{0,\nu}\).
\end{proof}

\section{Spectral method and convergence analysis}
\label{sec5}

This section constructs and analyzes a fractional backward spectral-collocation
method for the weakly singular adjoint Volterra integral equation
\begin{equation}\label{eq:3.1R}
u(t)
=
g(t)+(\mathcal{K}_R u)(t),
\qquad t\in J:=[0,1],
\end{equation}
where \(0<\theta<1\), \(g\in C(J)\), and \(\mathcal{K}_R\) is the weakly
singular adjoint Volterra integral operator
\begin{equation}\label{eq:KR-def-main}
(\mathcal{K}_R v)(t)
=
\int_t^1(\varrho-t)^{-\theta}K(t,\varrho)v(\varrho)\,d\varrho .
\end{equation}
Here \(K\in C(J\times J)\), \(K(t,t)\ne 0\) for \(t\in J\), and
\(\theta\) denotes the weak-singularity exponent.

\subsection{Spectral-collocation method}
\label{subsec:RJSC-method}

The fractional backward Zaky--Jacobi spectral-collocation approximation seeks
\(u_N^\rho\in P_N^\rho(J)\) such that
\begin{equation}\label{eq:3.1R-collocation}
u_N^\rho(t_i)
=
g(t_i)+(\mathcal{K}_R u_N^\rho)(t_i),
\qquad 0\le i\le N,
\end{equation}
where the collocation points \(\{t_i\}_{i=0}^{N}\) are the zeros of
\[
\mathcal{P}_{N+1}^{\mu,\upsilon,\rho}(t).
\]

In general, the exact evaluation of \((\mathcal{K}_R\varphi)(t_i)\) is not
available in closed form. Therefore, the weakly singular integral term is
approximated by a Jacobi--Gauss quadrature rule adapted to the Abel-type
singularity. For each collocation point \(t_i\), introduce the transformation
\[
\varrho=\varrho_i(\eta)
:=
t_i+(1-t_i)\bigl(1-(1-\eta)^{1/\rho}\bigr),
\qquad 0\le \eta\le 1.
\]
Then, for any \(\varphi\in C(J)\),
\begin{align}
(\mathcal{K}_R\varphi)(t_i)
&=
\int_{t_i}^{1}
(\varrho-t_i)^{-\theta}
K(t_i,\varrho)\varphi(\varrho)\,d\varrho
\nonumber\\
&=
\frac{(1-t_i)^{1-\theta}}{\rho}
\int_0^1
\bigl(1-(1-\eta)^{1/\rho}\bigr)^{-\theta}
(1-\eta)^{1/\rho-1}
K(t_i,\varrho_i(\eta))
\varphi(\varrho_i(\eta))\,d\eta
\nonumber\\
&=
\frac{(1-t_i)^{1-\theta}}{\rho}
\int_0^1
\eta^{-\theta}(1-\eta)^{1/\rho-1}
\left(
\frac{1-(1-\eta)^{1/\rho}}{\eta}
\right)^{-\theta}
K(t_i,\varrho_i(\eta))
\varphi(\varrho_i(\eta))\,d\eta
\nonumber\\
&=
\frac{(1-t_i)^{1-\theta}}{\rho}
\left(
\left(
\frac{1-(1-\eta)^{1/\rho}}{\eta}
\right)^{-\theta}
K(t_i,\varrho_i(\eta)),
\varphi(\varrho_i(\eta))
\right)_{\varkappa^{1/\rho-1,-\theta,1}} .
\label{eq:3.2R-transform}
\end{align}
Here
\[
\varkappa^{1/\rho-1,-\theta,1}(\eta)
=
(1-\eta)^{1/\rho-1}\eta^{-\theta}.
\]
Since \(0<\theta<1\) and \(0<\rho\le 1\), this is an admissible Jacobi
weight.

Define
\begin{equation}\label{eq:3.2R}
\overline{K}_R(t_i,\varrho_i(\eta))
=
\frac{(1-t_i)^{1-\theta}}{\rho}
\left(
\frac{1-(1-\eta)^{1/\rho}}{\eta}
\right)^{-\theta}
K(t_i,\varrho_i(\eta)).
\end{equation}
Then
\begin{equation}\label{eq:3.2R-inner}
(\mathcal{K}_R\varphi)(t_i)
=
\left(
\overline{K}_R(t_i,\varrho_i(\eta)),
\varphi(\varrho_i(\eta))
\right)_{\varkappa^{1/\rho-1,-\theta,1}} .
\end{equation}

We approximate this weighted inner product by the Jacobi--Gauss quadrature rule
\begin{equation}\label{eq:3.3R}
\left(
\overline{K}_R(t_i,\varrho_i(\eta)),
\varphi(\varrho_i(\eta))
\right)_{N,\varkappa^{1/\rho-1,-\theta,1}}
:=
\sum_{j=0}^{N}
\overline{K}_R(t_i,\varrho_i(\eta_j))
\varphi(\varrho_i(\eta_j))\chi_j ,
\end{equation}
where \(\{\eta_j\}_{j=0}^{N}\) are the zeros of
\[
\mathcal{P}_{N+1}^{1/\rho-1,-\theta,1}(\eta),
\]
and \(\{\chi_j\}_{j=0}^{N}\) are the corresponding quadrature weights.

For brevity, define
\begin{equation}\label{eq:3.4R}
(\mathcal{K}_{R,N}\varphi)(t_i)
:=
\left(
\overline{K}_R(t_i,\varrho_i(\eta)),
\varphi(\varrho_i(\eta))
\right)_{N,\varkappa^{1/\rho-1,-\theta,1}} .
\end{equation}
Equivalently,
\begin{equation}\label{eq:3.4R-expanded}
(\mathcal{K}_{R,N}\varphi)(t_i)
=
\sum_{j=0}^{N}
\overline{K}_R(t_i,\varrho_i(\eta_j))
\varphi(\varrho_i(\eta_j))\chi_j .
\end{equation}

The fully discrete fractional backward collocation scheme is therefore
formulated as follows: find
\[
u_N^\rho(t)
:=
\sum_{i=0}^{N}u_i h_{i,\rho}^{\mu,\upsilon}(t)
\in P_N^\rho(J)
\]
such that
\begin{equation}\label{eq:3.5R}
u_N^\rho(t_i)
=
g(t_i)+(\mathcal{K}_{R,N}u_N^\rho)(t_i),
\qquad 0\le i\le N,
\end{equation}
where \(\{h_{i,\rho}^{\mu,\upsilon}\}_{i=0}^{N}\) are the generalized
Lagrange basis functions defined in \eqref{eq:2.22R}. Since
\(u_N^\rho(t_i)=u_i\), the scheme is equivalently written as the algebraic
system
\begin{equation}\label{eq:3.5R-algebraic}
u_i
=
g(t_i)
+
\sum_{j=0}^{N}
\overline{K}_R(t_i,\varrho_i(\eta_j))
u_N^\rho(\varrho_i(\eta_j))\chi_j,
\qquad 0\le i\le N .
\end{equation}

\subsection{Convergence analysis and error estimates}
\label{subsec:RJSC-convergence-analysis}

We now analyze the discrete problem \eqref{eq:3.5R} and derive error estimates
for \(u-u_N^\rho\). The first theorem establishes convergence in the uniform
norm.

\begin{theorem}\label{thm:3.1R}
Let \(u\) be the exact solution of \eqref{eq:3.1R}, and let \(u_N^\rho\) be
the numerical solution of the Zaky--Jacobi collocation problem
\eqref{eq:3.5R}. Assume that \(0<\theta<1\),
\(-1<\mu,\upsilon\le -\frac12\), \(K\in C^m(J\times J)\), and
\[
u\bigl(1-(1-t)^{1/\rho}\bigr)\in B_{\mu,\upsilon}^{m,1}(J),
\qquad m\ge 1.
\]
Then, for \(N\) sufficiently large, there exists a positive constant \(c\),
independent of \(N\), such that
\begin{equation}\label{eq:3.6R}
\|u-u_N^\rho\|_{\infty}
\le
cN^{\frac12-m}
\left(
\left\|
\partial_t^m
\bigl\{u\bigl(1-(1-t)^{1/\rho}\bigr)\bigr\}
\right\|_{0,\varkappa^{\mu+m,\upsilon+m,1}}
+
N^{-\frac12}\log N\,K_R^*\|u\|_{\infty}
\right),
\end{equation}
where
\begin{equation}\label{eq:3.7R}
K_R^*
=
\max_{0\le i\le N}
\left\|
\partial_\eta^m
\overline{K}_R(t_i,\varrho_i(\cdot))
\right\|_{0,\varkappa^{m+1/\rho-1,m-\theta,1}} .
\end{equation}
\end{theorem}

\begin{proof}
Let
\[
\mathcal{E}(t)=u(t)-u_N^\rho(t)
\]
denote the error. Subtracting the discrete collocation equation
\eqref{eq:3.5R} from the continuous equation \eqref{eq:3.1R} at the
collocation points \(\{t_i\}_{i=0}^{N}\) gives
\begin{equation}\label{eq:3.8R}
\mathcal{E}_i
=
(\mathcal{K}_R \mathcal{E})(t_i)+q_i,
\qquad 0\le i\le N,
\end{equation}
where
\[
\mathcal{E}_i=u(t_i)-u_i,
\]
and \(q_i\) denotes the quadrature residual
\begin{align}
q_i
&=
(\mathcal{K}_R u_N^\rho)(t_i)
-
(\mathcal{K}_{R,N}u_N^\rho)(t_i)
\nonumber\\
&=
\left(
\overline{K}_R(t_i,\varrho_i(\eta)),
u_N^\rho(\varrho_i(\eta))
\right)_{\varkappa^{1/\rho-1,-\theta,1}}
-
\left(
\overline{K}_R(t_i,\varrho_i(\eta)),
u_N^\rho(\varrho_i(\eta))
\right)_{N,\varkappa^{1/\rho-1,-\theta,1}} .
\label{eq:q_i_R}
\end{align}
By the Jacobi--Gauss quadrature error estimate in Lemma~\ref{lem:2.9R}, we
obtain
\begin{equation}\label{eq:3.9R}
|q_i|
\le
cN^{-m}
\left\|
\partial_\eta^m
\overline{K}_R(t_i,\varrho_i(\cdot))
\right\|_{0,\varkappa^{m+1/\rho-1,m-\theta,1}}
\left\|
u_N^\rho(\varrho_i(\cdot))
\right\|_{0,\varkappa^{1/\rho-1,-\theta,1}} .
\end{equation}

Multiplying \eqref{eq:3.8R} by \(h_{i,\rho}^{\mu,\upsilon}(t)\) and summing
over \(i=0,\ldots,N\), we obtain
\begin{equation}\label{eq:3.10R}
J_{N,\rho}^{\mu,\upsilon}u(t)-u_N^\rho(t)
=
J_{N,\rho}^{\mu,\upsilon}(\mathcal{K}_R \mathcal{E})(t)
+
\sum_{i=0}^{N}q_i h_{i,\rho}^{\mu,\upsilon}(t).
\end{equation}
Adding and subtracting \(u(t)\) and \((\mathcal{K}_R \mathcal{E})(t)\), this
identity can be rewritten as
\begin{equation}\label{eq:3.11R}
\mathcal{E}(t)
=
(\mathcal{K}_R \mathcal{E})(t)+R_1+R_2+R_3,
\end{equation}
where
\begin{equation}\label{eq:3.12R}
R_1
=
u(t)-J_{N,\rho}^{\mu,\upsilon}u(t),
\qquad
R_2
=
\sum_{i=0}^{N}q_i h_{i,\rho}^{\mu,\upsilon}(t),
\qquad
R_3
=
J_{N,\rho}^{\mu,\upsilon}(\mathcal{K}_R \mathcal{E})(t)
-
(\mathcal{K}_R \mathcal{E})(t).
\end{equation}
Thus, \(R_1\) is the interpolation error of \(u\), \(R_2\) is the accumulated
quadrature error, and \(R_3\) is the interpolation error associated with the
integral term \(\mathcal{K}_R \mathcal{E}\).

From \eqref{eq:3.11R}, using the boundedness of \(K\), we have
\begin{align}
|\mathcal{E}(t)|
&\le
|R_1+R_2+R_3|
+
\int_t^1(\varrho-t)^{-\theta}|K(t,\varrho)|\,|\mathcal{E}(\varrho)|\,d\varrho
\nonumber\\
&\le
|R_1+R_2+R_3|
+
K_0\int_t^1(\varrho-t)^{-\theta}|\mathcal{E}(\varrho)|\,d\varrho,
\label{eq:3.13R-a}
\end{align}
where
\[
K_0=\max_{0\le t\le \varrho\le 1}|K(t,\varrho)|.
\]
Applying the backward Gronwall inequality in Lemma~\ref{lem:right-gronwall}
to \eqref{eq:3.13R-a} gives
\begin{align}
|\mathcal{E}(t)|
&\le
|R_1+R_2+R_3|
+
K_0
\int_t^1
(\varrho-t)^{-\theta}
|R_1+R_2+R_3|
\exp\left(
K_0\int_t^\varrho(\xi-t)^{-\theta}\,d\xi
\right)d\varrho
\nonumber\\
&\le
|R_1+R_2+R_3|
+
K_0
\exp\left(
K_0\int_t^1(\xi-t)^{-\theta}\,d\xi
\right)
\int_t^1
(\varrho-t)^{-\theta}
|R_1+R_2+R_3|\,d\varrho
\nonumber\\
&\le
|R_1+R_2+R_3|
+
K_0
\exp\left(\frac{K_0}{1-\theta}\right)
\int_t^1
(\varrho-t)^{-\theta}
|R_1+R_2+R_3|\,d\varrho .
\label{eq:3.13R}
\end{align}
Consequently,
\begin{equation}\label{eq:3.14R}
\|\mathcal{E}\|_{\infty}
\le
c\left(
\|R_1\|_{\infty}
+
\|R_2\|_{\infty}
+
\|R_3\|_{\infty}
\right),
\end{equation}
where
\[
c
=
1+\frac{K_0}{1-\theta}
\exp\left(\frac{K_0}{1-\theta}\right).
\]

We now estimate the three terms on the right-hand side of
\eqref{eq:3.14R}. First, by the \(L^\infty\)-interpolation estimate in
Proposition~\ref{prop:2.5R}, we obtain
\begin{equation}\label{eq:3.15R}
\|R_1\|_{\infty}
=
\|u-J_{N,\rho}^{\mu,\upsilon}u\|_{\infty}
\le
cN^{\frac12-m}
\left\|
\partial_t^m
\bigl\{u\bigl(1-(1-t)^{1/\rho}\bigr)\bigr\}
\right\|_{0,\varkappa^{\mu+m,\upsilon+m,1}} .
\end{equation}

For \(R_2\), estimate \eqref{eq:3.9R} yields
\begin{align}
\max_{0\le i\le N}|q_i|
&\le
cN^{-m}
\max_{0\le i\le N}
\left\|
\partial_\eta^m
\overline{K}_R(t_i,\varrho_i(\cdot))
\right\|_{0,\varkappa^{m+1/\rho-1,m-\theta,1}}
\max_{0\le i\le N}
\left\|
u_N^\rho(\varrho_i(\cdot))
\right\|_{0,\varkappa^{1/\rho-1,-\theta,1}}
\nonumber\\
&\le
cN^{-m}K_R^*
\max_{0\le i\le N}
\left\|
u_N^\rho(\varrho_i(\cdot))
\right\|_{0,\varkappa^{1/\rho-1,-\theta,1}} .
\label{eq:3.16R-a}
\end{align}
Since the Jacobi weight \(\varkappa^{1/\rho-1,-\theta,1}\) is integrable on
\(J\), we have
\[
\left\|
u_N^\rho(\varrho_i(\cdot))
\right\|_{0,\varkappa^{1/\rho-1,-\theta,1}}
\le
c\|u_N^\rho\|_{\infty}.
\]
Hence,
\begin{align}
\max_{0\le i\le N}|q_i|
&\le
cN^{-m}K_R^*
\|u_N^\rho\|_{\infty}
\nonumber\\
&\le
cN^{-m}K_R^*
\bigl(\|\mathcal{E}\|_{\infty}+\|u\|_{\infty}\bigr).
\label{eq:3.16R}
\end{align}
Using the Lebesgue-constant estimate in Lemma~\ref{lem:2.11R}, we get
\begin{align}
\|R_2\|_{\infty}
&=
\left\|
\sum_{i=0}^{N}q_i h_{i,\rho}^{\mu,\upsilon}(t)
\right\|_{\infty}
\nonumber\\
&=
\left\|
\sum_{i=0}^{N}q_i
h_{i,1}^{\mu,\upsilon}\bigl(1-(1-t)^\rho\bigr)
\right\|_{\infty}
\nonumber\\
&\le
\max_{0\le i\le N}|q_i|
\max_{t\in J}
\sum_{i=0}^{N}
\left|
h_{i,1}^{\mu,\upsilon}\bigl(1-(1-t)^\rho\bigr)
\right|
\nonumber\\
&\le
c\log N\,\max_{0\le i\le N}|q_i|
\nonumber\\
&\le
cN^{-m}\log N\,K_R^*
\bigl(\|\mathcal{E}\|_{\infty}+\|u\|_{\infty}\bigr).
\label{eq:3.17R}
\end{align}

It remains to bound \(R_3\). Using the relation between the Zaky--Jacobi
interpolation operator and the shifted Jacobi interpolation operator, given in
\eqref{eq:2.24R-a}, we obtain
\begin{align}
\|R_3\|_{\infty}
&=
\max_{t\in J}
\left|
J_{N,\rho}^{\mu,\upsilon}(\mathcal{K}_R \mathcal{E})(t)
-
(\mathcal{K}_R \mathcal{E})(t)
\right|
\nonumber\\
&=
\max_{z\in J}
\left|
J_{N,1}^{\mu,\upsilon}
\left[
(\mathcal{K}_R \mathcal{E})\bigl(1-(1-z)^{1/\rho}\bigr)
\right]
-
(\mathcal{K}_R \mathcal{E})\bigl(1-(1-z)^{1/\rho}\bigr)
\right|
\nonumber\\
&=
\left\|
\left(J_{N,1}^{\mu,\upsilon}-I\right)
\left[
(\mathcal{K}_R \mathcal{E})\bigl(1-(1-z)^{1/\rho}\bigr)
\right]
\right\|_{\infty}.
\label{eq:3.18R-a}
\end{align}
Let
\[
F_{\mathcal{E}}(z):=(\mathcal{K}_R \mathcal{E})\bigl(1-(1-z)^{1/\rho}\bigr).
\]
Since \(\mathscr{T}_N F_{\mathcal{E}}\in P_N^1(J)\), the interpolation operator satisfies
\[
J_{N,1}^{\mu,\upsilon}\mathscr{T}_N F_{\mathcal{E}}=\mathscr{T}_N F_{\mathcal{E}}.
\]
Therefore,
\begin{align}
\|R_3\|_{\infty}
&=
\left\|
\left(J_{N,1}^{\mu,\upsilon}-I\right)
F_{\mathcal{E}}
\right\|_{\infty}
\nonumber\\
&=
\left\|
\left(J_{N,1}^{\mu,\upsilon}-I\right)
\left(F_{\mathcal{E}}-\mathscr{T}_N F_{\mathcal{E}}\right)
\right\|_{\infty}
\nonumber\\
&\le
\left(
\|J_{N,1}^{\mu,\upsilon}\|_{\infty}+1
\right)
\left\|
F_{\mathcal{E}}-\mathscr{T}_N F_{\mathcal{E}}
\right\|_{\infty}
\nonumber\\
&\le
c\log N\,
\left\|
F_{\mathcal{E}}-\mathscr{T}_N F_{\mathcal{E}}
\right\|_{\infty}.
\label{eq:3.18R-b}
\end{align}
By the Ragozin approximation estimate \eqref{eq:Ragozin-R}, with \(m=0\), and
by Lemma~\ref{lem:KR-holder-transform}, we have
\begin{align}
\left\|
F_{\mathcal{E}}-\mathscr{T}_N F_{\mathcal{E}}
\right\|_{\infty}
&\le
cN^{-\nu}\|F_{\mathcal{E}}\|_{0,\nu}
\nonumber\\
&=
cN^{-\nu}
\left\|
(\mathcal{K}_R \mathcal{E})\bigl(1-(1-z)^{1/\rho}\bigr)
\right\|_{0,\nu}
\nonumber\\
&\le
cN^{-\nu}\|\mathcal{E}\|_{\infty},
\qquad 0<\nu<1-\theta .
\label{eq:3.18R-c}
\end{align}
Combining \eqref{eq:3.18R-b} and \eqref{eq:3.18R-c} yields
\begin{equation}\label{eq:3.18R}
\|R_3\|_{\infty}
\le
cN^{-\nu}\log N\,\|\mathcal{E}\|_{\infty},
\qquad 0<\nu<1-\theta .
\end{equation}

Substituting \eqref{eq:3.15R}, \eqref{eq:3.17R}, and \eqref{eq:3.18R} into
\eqref{eq:3.14R}, we arrive at
\begin{align}
\|\mathcal{E}\|_{\infty}
&\le
cN^{\frac12-m}
\left\|
\partial_t^m
\bigl\{u\bigl(1-(1-t)^{1/\rho}\bigr)\bigr\}
\right\|_{0,\varkappa^{\mu+m,\upsilon+m,1}}
\nonumber\\
&\quad
+
cN^{-m}\log N\,K_R^*
\bigl(\|\mathcal{E}\|_{\infty}+\|u\|_{\infty}\bigr)
+
cN^{-\nu}\log N\,\|\mathcal{E}\|_{\infty}.
\label{eq:3.18R-d}
\end{align}
Equivalently,
\begin{align}
\left(
1
-
cN^{-m}\log N\,K_R^*
-
cN^{-\nu}\log N
\right)
\|\mathcal{E}\|_{\infty}
&\le
cN^{\frac12-m}
\left\|
\partial_t^m
\bigl\{u\bigl(1-(1-t)^{1/\rho}\bigr)\bigr\}
\right\|_{0,\varkappa^{\mu+m,\upsilon+m,1}}
\nonumber\\
&\quad
+
cN^{-m}\log N\,K_R^*\|u\|_{\infty}.
\label{eq:3.18R-e}
\end{align}
For \(N\) sufficiently large, the coefficient on the left-hand side is
bounded below by a positive constant. Hence,
\begin{align}
\|\mathcal{E}\|_{\infty}
&\le
cN^{\frac12-m}
\left\|
\partial_t^m
\bigl\{u\bigl(1-(1-t)^{1/\rho}\bigr)\bigr\}
\right\|_{0,\varkappa^{\mu+m,\upsilon+m,1}}
+
cN^{-m}\log N\,K_R^*\|u\|_{\infty}
\nonumber\\
&=
cN^{\frac12-m}
\left(
\left\|
\partial_t^m
\bigl\{u\bigl(1-(1-t)^{1/\rho}\bigr)\bigr\}
\right\|_{0,\varkappa^{\mu+m,\upsilon+m,1}}
+
N^{-\frac12}\log N\,K_R^*\|u\|_{\infty}
\right).
\end{align}
This proves \eqref{eq:3.6R}.
\end{proof}

To derive the weighted \(L^2\)-error estimate, we use the following generalized
Hardy inequality with weights \cite{gogatishvili1999generalized,persson2017weighted}.

\begin{lemma}\label{lem:3.1R}
Let \(f\ge 0\) be measurable, and let \(U\) and \(V\) be weight functions.
Assume that \(1<p\le q<\infty\) and \(-\infty\le a<b\le\infty\). Then the
right-sided generalized Hardy inequality
\begin{equation}\label{eq:Hardy-R}
\left(
\int_a^b |(T_R f)(t)|^q U(t)\,dt
\right)^{1/q}
\le
c
\left(
\int_a^b |f(t)|^p V(t)\,dt
\right)^{1/p}
\end{equation}
is valid if and only if
\begin{equation}\label{eq:Hardy-cond-R}
\sup_{a<t<b}
\left(
\int_a^t U(x)\,dx
\right)^{1/q}
\left(
\int_t^b V^{1-p'}(x)\,dx
\right)^{1/p'}
<\infty,
\qquad
p'=\frac{p}{p-1},
\end{equation}
where
\begin{equation}\label{eq:Hardy-operator-R}
(T_R f)(t)
=
\int_t^b \rho(t,\varrho)f(\varrho)\,d\varrho
\end{equation}
for a prescribed nonnegative kernel \(\rho(t,\varrho)\).
\end{lemma}

Using the weighted mean convergence property of Jacobi interpolation
\cite{chen2013note}, together with the transformation identity
\eqref{eq:2.26R}, we obtain the following stability estimate for the
Zaky--Jacobi interpolation operator.

\begin{lemma}\label{lem:3.2R}
For any bounded function \(v\) on \(J\), there exists a positive constant
\(c\), independent of \(v\), such that
\begin{equation}\label{eq:3.19R-lemma}
\sup_N
\|J_{N,\rho}^{\mu,\upsilon}v\|_{0,\varkappa^{\mu,\upsilon,\rho}}
\le
c\|v\|_{\infty}.
\end{equation}
\end{lemma}

\begin{theorem}\label{thm:3.2R}
Let \(u\) denote the exact solution of \eqref{eq:3.1R}, and let \(u_N^\rho\)
be the solution of the discrete scheme \eqref{eq:3.5R}. Suppose that
\(0<\theta<1\), \(-1<\mu,\upsilon\le -\frac12\),
\(K\in C^m(J\times J)\), and
\[
u\bigl(1-(1-t)^{1/\rho}\bigr)\in B_{\mu,\upsilon}^{m,1}(J),
\qquad m\ge 1.
\]
Then, for  sufficiently large \(N\), there exists a positive constant \(c\),
independent of \(N\), such that
\begin{equation}\label{eq:3.19R}
\|u-u_N^\rho\|_{0,\varkappa^{\mu,\upsilon,\rho}}
\le
cN^{-m}
\left[
\left(1+N^{\frac12-\nu}\right)
\left\|
\partial_t^m
\bigl\{u\bigl(1-(1-t)^{1/\rho}\bigr)\bigr\}
\right\|_{0,\varkappa^{\mu+m,\upsilon+m,1}}
+
K_R^*\|u\|_{\infty}
\right],
\qquad
0<\nu<1-\theta,
\end{equation}
where \(K_R^*\) is defined in \eqref{eq:3.7R}.
\end{theorem}

\begin{proof}
The error representation is given by
\[
\mathcal{E}(t)
=
(\mathcal{K}_R \mathcal{E})(t)+R_1+R_2+R_3.
\]
Taking the weighted norm and applying the right-sided Hardy inequality from
Lemma~\ref{lem:3.1R}, we obtain
\begin{align}
\|\mathcal{E}\|_{0,\varkappa^{\mu,\upsilon,\rho}}
&\le
\|R_1\|_{0,\varkappa^{\mu,\upsilon,\rho}}
+
\|R_2\|_{0,\varkappa^{\mu,\upsilon,\rho}}
+
\|R_3\|_{0,\varkappa^{\mu,\upsilon,\rho}}
\nonumber\\
&\quad
+
K_0
\exp\left(\frac{K_0}{1-\theta}\right)
\left\|
\int_t^1
(\varrho-t)^{-\theta}
|R_1+R_2+R_3|\,d\varrho
\right\|_{0,\varkappa^{\mu,\upsilon,\rho}}
\nonumber\\
&\le
c\left(
\|R_1\|_{0,\varkappa^{\mu,\upsilon,\rho}}
+
\|R_2\|_{0,\varkappa^{\mu,\upsilon,\rho}}
+
\|R_3\|_{0,\varkappa^{\mu,\upsilon,\rho}}
\right).
\label{eq:3.20R}
\end{align}

We now estimate each contribution separately. By Proposition~\ref{prop:2.4R},
\begin{equation}\label{eq:3.21R}
\|R_1\|_{0,\varkappa^{\mu,\upsilon,\rho}}
=
\|u-J_{N,\rho}^{\mu,\upsilon}u\|_{0,\varkappa^{\mu,\upsilon,\rho}}
\le
cN^{-m}
\left\|
\partial_t^m
\bigl\{u\bigl(1-(1-t)^{1/\rho}\bigr)\bigr\}
\right\|_{0,\varkappa^{\mu+m,\upsilon+m,1}} .
\end{equation}

Next, Lemma~\ref{lem:3.2R}, together with \eqref{eq:3.16R}, yields
\begin{align}
\|R_2\|_{0,\varkappa^{\mu,\upsilon,\rho}}
&=
\left\|
\sum_{i=0}^{N}q_i
h_{i,\rho}^{\mu,\upsilon}(t)
\right\|_{0,\varkappa^{\mu,\upsilon,\rho}}
\nonumber\\
&=
\left\|
\sum_{i=0}^{N}q_i
h_{i,1}^{\mu,\upsilon}\bigl(1-(1-t)^\rho\bigr)
\right\|_{0,\varkappa^{\mu,\upsilon,\rho}}
\nonumber\\
&\le
c\max_{0\le i\le N}|q_i|
\nonumber\\
&\le
cN^{-m}K_R^*
\bigl(\|\mathcal{E}\|_{\infty}+\|u\|_{\infty}\bigr).
\label{eq:3.22R}
\end{align}

It remains to bound \(R_3\). Employing the change of variables
\(z=1-(1-t)^\rho\), we have
\begin{align}
&
\left\|
\left(J_{N,\rho}^{\mu,\upsilon}-I\right)
(\mathcal{K}_R \mathcal{E})(t)
\right\|_{0,\varkappa^{\mu,\upsilon,\rho}}
\nonumber\\
&=
\left\{
\int_0^1
\left[
\left(J_{N,\rho}^{\mu,\upsilon}-I\right)
(\mathcal{K}_R \mathcal{E})(t)
\right]^2
\rho(1-t)^{\rho(\mu+1)-1}
\bigl(1-(1-t)^\rho\bigr)^\upsilon
\,dt
\right\}^{1/2}
\nonumber\\
&=
\left\{
\int_0^1
\left[
\sum_{i=0}^{N}
(\mathcal{K}_R \mathcal{E})(t_i)
h_{i,\rho}^{\mu,\upsilon}(t)
-
(\mathcal{K}_R \mathcal{E})(t)
\right]^2
(1-z)^\mu z^\upsilon\,dz
\right\}^{1/2}
\nonumber\\
&=
\left\{
\int_0^1
\left[
\sum_{i=0}^{N}
(\mathcal{K}_R \mathcal{E})\bigl(1-(1-z_i)^{1/\rho}\bigr)
h_{i,1}^{\mu,\upsilon}(z)
-
(\mathcal{K}_R \mathcal{E})\bigl(1-(1-z)^{1/\rho}\bigr)
\right]^2
(1-z)^\mu z^\upsilon\,dz
\right\}^{1/2}
\nonumber\\
&=
\left\|
\left(J_{N,1}^{\mu,\upsilon}-I\right)
\left[
(\mathcal{K}_R \mathcal{E})\bigl(1-(1-z)^{1/\rho}\bigr)
\right]
\right\|_{0,\varkappa^{\mu,\upsilon,1}},
\label{eq:3.23R-transform}
\end{align}
where \(\{z_i\}_{i=0}^{N}\) are the zeros of
\(\mathcal{P}_{N+1}^{\mu,\upsilon,1}(z)\).

Set again
\[
F_{\mathcal{E}}(z):=(\mathcal{K}_R \mathcal{E})\bigl(1-(1-z)^{1/\rho}\bigr).
\]
Since \(\mathscr{T}_N F_{\mathcal{E}}\in P_N^1(J)\), we have
\[
J_{N,1}^{\mu,\upsilon}\mathscr{T}_N F_{\mathcal{E}}
=
\mathscr{T}_N F_{\mathcal{E}}.
\]
Therefore,
\begin{align}
\|R_3\|_{0,\varkappa^{\mu,\upsilon,\rho}}
&=
\left\|
\left(J_{N,1}^{\mu,\upsilon}-I\right)F_{\mathcal{E}}
\right\|_{0,\varkappa^{\mu,\upsilon,1}}
\nonumber\\
&=
\left\|
\left(J_{N,1}^{\mu,\upsilon}-I\right)
\left(F_{\mathcal{E}}-\mathscr{T}_N F_{\mathcal{E}}\right)
\right\|_{0,\varkappa^{\mu,\upsilon,1}}
\nonumber\\
&\le
\left\|
J_{N,1}^{\mu,\upsilon}
\left(F_{\mathcal{E}}-\mathscr{T}_N F_{\mathcal{E}}\right)
\right\|_{0,\varkappa^{\mu,\upsilon,1}}
+
\left\|
F_{\mathcal{E}}-\mathscr{T}_N F_{\mathcal{E}}
\right\|_{0,\varkappa^{\mu,\upsilon,1}}
\nonumber\\
&\le
c
\left\|
F_{\mathcal{E}}-\mathscr{T}_N F_{\mathcal{E}}
\right\|_{\infty}.
\label{eq:3.23R-a}
\end{align}
Using the Ragozin approximation estimate and Lemma~\ref{lem:KR-holder-transform},
we obtain
\begin{align}
\left\|
F_{\mathcal{E}}-\mathscr{T}_N F_{\mathcal{E}}
\right\|_{\infty}
&\le
cN^{-\nu}\|F_{\mathcal{E}}\|_{0,\nu}
\nonumber\\
&=
cN^{-\nu}
\left\|
(\mathcal{K}_R \mathcal{E})\bigl(1-(1-z)^{1/\rho}\bigr)
\right\|_{0,\nu}
\nonumber\\
&\le
cN^{-\nu}\|\mathcal{E}\|_{\infty},
\qquad 0<\nu<1-\theta .
\label{eq:3.23R-b}
\end{align}
Combining \eqref{eq:3.23R-a} and \eqref{eq:3.23R-b}, we obtain
\begin{equation}\label{eq:3.23R}
\|R_3\|_{0,\varkappa^{\mu,\upsilon,\rho}}
\le
cN^{-\nu}\|\mathcal{E}\|_{\infty},
\qquad 0<\nu<1-\theta .
\end{equation}
Using Theorem~\ref{thm:3.1R}, this further gives
\begin{align}
\|R_3\|_{0,\varkappa^{\mu,\upsilon,\rho}}
&\le
cN^{-\nu}
N^{\frac12-m}
\left(
\left\|
\partial_t^m
\bigl\{u\bigl(1-(1-t)^{1/\rho}\bigr)\bigr\}
\right\|_{0,\varkappa^{\mu+m,\upsilon+m,1}}
+
N^{-\frac12}\log N\,K_R^*\|u\|_{\infty}
\right)
\nonumber\\
&=
cN^{\frac12-m-\nu}
\left(
\left\|
\partial_t^m
\bigl\{u\bigl(1-(1-t)^{1/\rho}\bigr)\bigr\}
\right\|_{0,\varkappa^{\mu+m,\upsilon+m,1}}
+
N^{-\frac12}\log N\,K_R^*\|u\|_{\infty}
\right).
\label{eq:3.24R}
\end{align}

Substituting \eqref{eq:3.21R}, \eqref{eq:3.22R}, and \eqref{eq:3.24R} into
\eqref{eq:3.20R}, and using Theorem~\ref{thm:3.1R} to control
\(\|\mathcal{E}\|_\infty\) in \eqref{eq:3.22R}, we arrive at
\[
\|u-u_N^\rho\|_{0,\varkappa^{\mu,\upsilon,\rho}}
\le
cN^{-m}
\left[
\left(1+N^{\frac12-\nu}\right)
\left\|
\partial_t^m
\bigl\{u\bigl(1-(1-t)^{1/\rho}\bigr)\bigr\}
\right\|_{0,\varkappa^{\mu+m,\upsilon+m,1}}
+
K_R^*\|u\|_{\infty}
\right],
\]
for every \(0<\nu<1-\theta\). This proves \eqref{eq:3.19R}.
\end{proof}

\section{Numerical results}
\label{sec6}

This section reports numerical experiments designed to assess the accuracy and
convergence behaviour of the proposed fractional backward Zaky--Jacobi
spectral-collocation method. The numerical errors are measured in the
\(L^{\infty}(J)\)-norm and in weighted \(L^{2}\)-norms induced by the
Zaky--Jacobi weights.

\begin{example}\label{ex:right-sided-example}
Consider the exact solution
\[
u(t)=(1-t)^{-\theta}\sin(1-t),
\]
with \(K(t,\varrho)=1\). The corresponding source term is given by
\begin{equation}\label{eq:right-sided-g}
\begin{aligned}
g(t)
&=
(1-t)^{-\theta}\sin(1-t)
-
\sqrt{\pi}\Gamma(1-\theta)
(1-t)^{\frac{1}{2}-\theta}
\sin\left(\frac{1-t}{2}\right)
\mathcal{B}\left(\frac{1}{2}-\theta,\frac{1-t}{2}\right),
\end{aligned}
\end{equation}
where \(\mathcal{B}(\theta,t)=J_{\theta}(t)\) denotes the Bessel function of
the first kind, written as
\begin{equation}\label{eq:right-sided-bessel}
\mathcal{B}(\theta,t)
=
\left(\frac{t}{2}\right)^{\theta}
\sum_{i=0}^{+\infty}
\frac{(-t^{2})^{i}}
{i!\Gamma(\theta+i+1)4^{i}} .
\end{equation}

This exact solution has a weak singularity at the terminal endpoint \(t=1\).

In Fig.~\ref{fig:right-sided-example}, the errors in the
\(L^{\infty}(J)\)- and
\(L^{2}_{\varkappa^{-1/4,-1/4,\rho}}(J)\)-norms are displayed on a
semi-logarithmic scale as functions of the \(\rho\)-polynomial degree \(N\),
for
\(\theta=\frac{1}{2},\frac{2}{3}\).
The computations are carried out with
\[
\rho=\frac{1}{2},\frac{1}{4}
\quad \text{for} \quad
\theta=\frac{1}{2},
\]
and
\[
\rho=\frac{1}{3},\frac{1}{6}
\quad \text{for} \quad
\theta=\frac{2}{3}.
\]

The obtained results demonstrate exponential convergence of the proposed
fractional backward Zaky--Jacobi spectral method. This behaviour is consistent
with the theoretical prediction, since for the selected values of \(\theta\)
and \(\rho\), the transformed function
\[
u\!\left(1-(1-t)^{1/\rho}\right)
\]
is smooth.

\begin{figure}[h!]
\centering
 \includegraphics[width=0.48\textwidth]{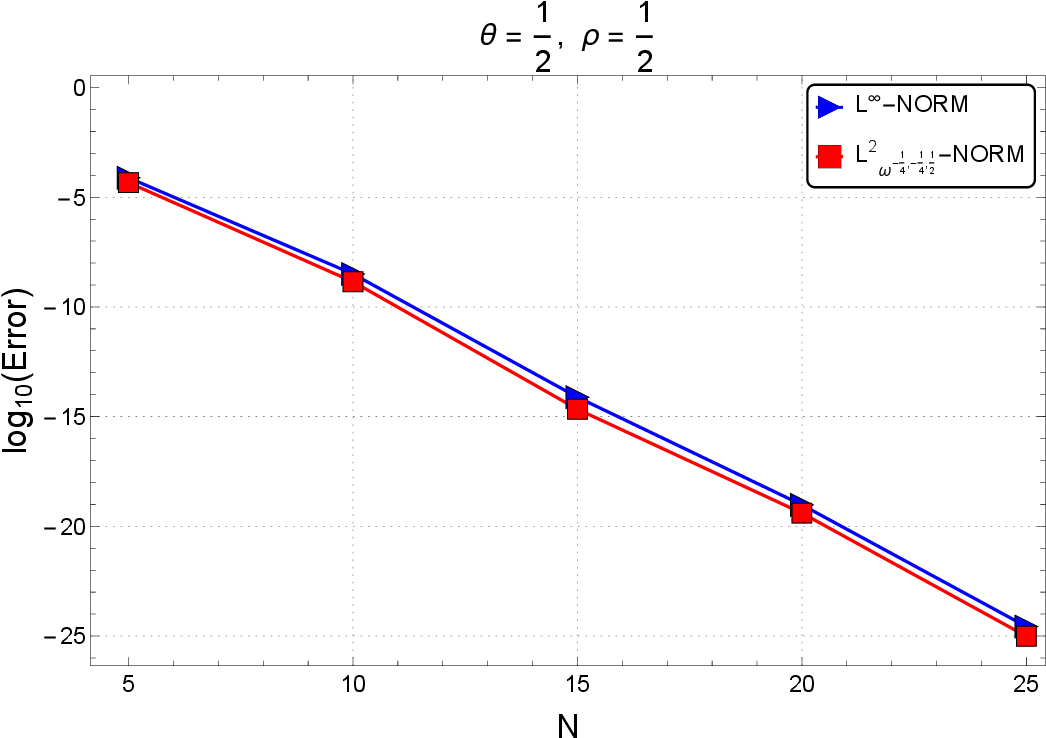}
 \includegraphics[width=0.48\textwidth]{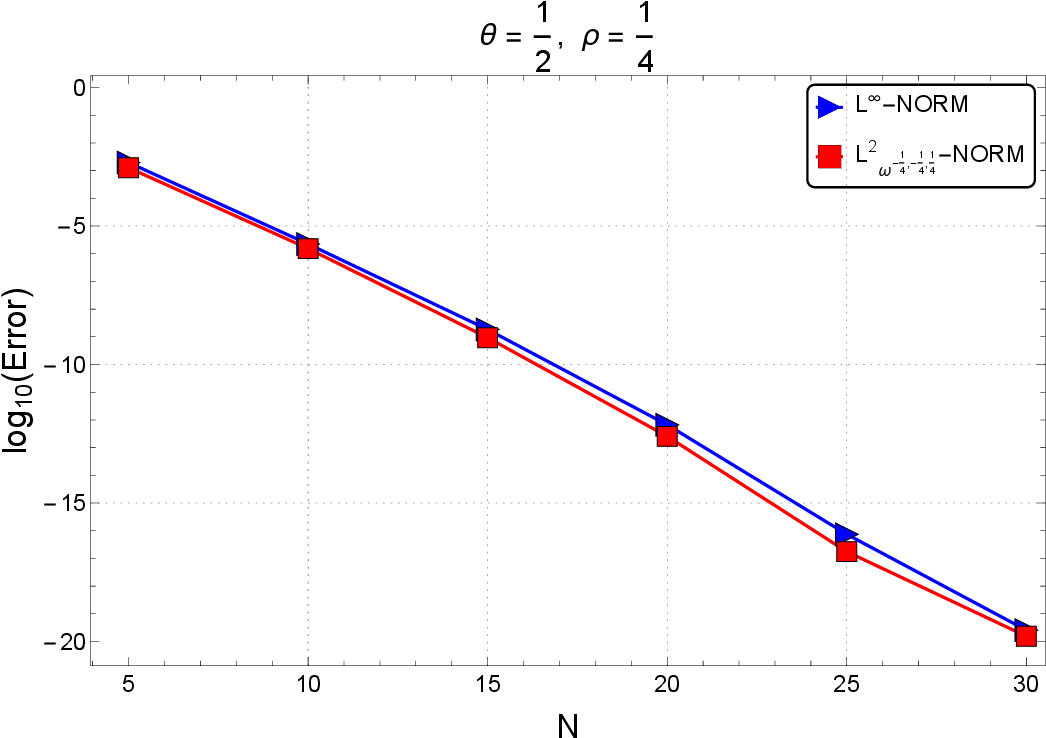}
 \includegraphics[width=0.48\textwidth]{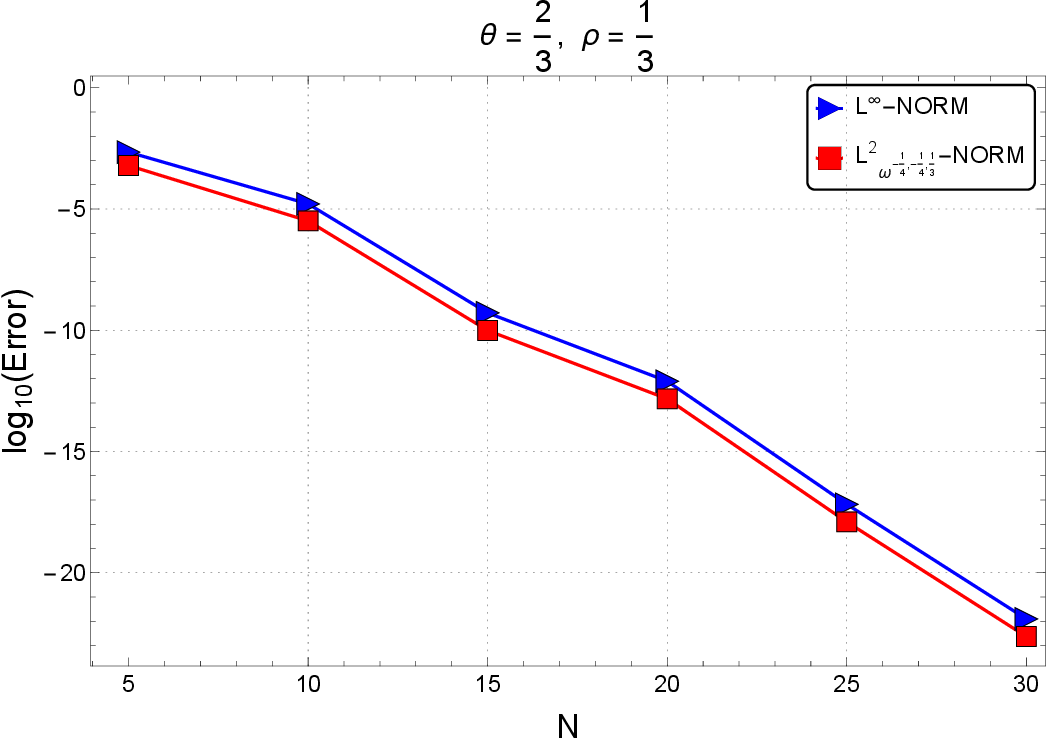}
 \includegraphics[width=0.48\textwidth]{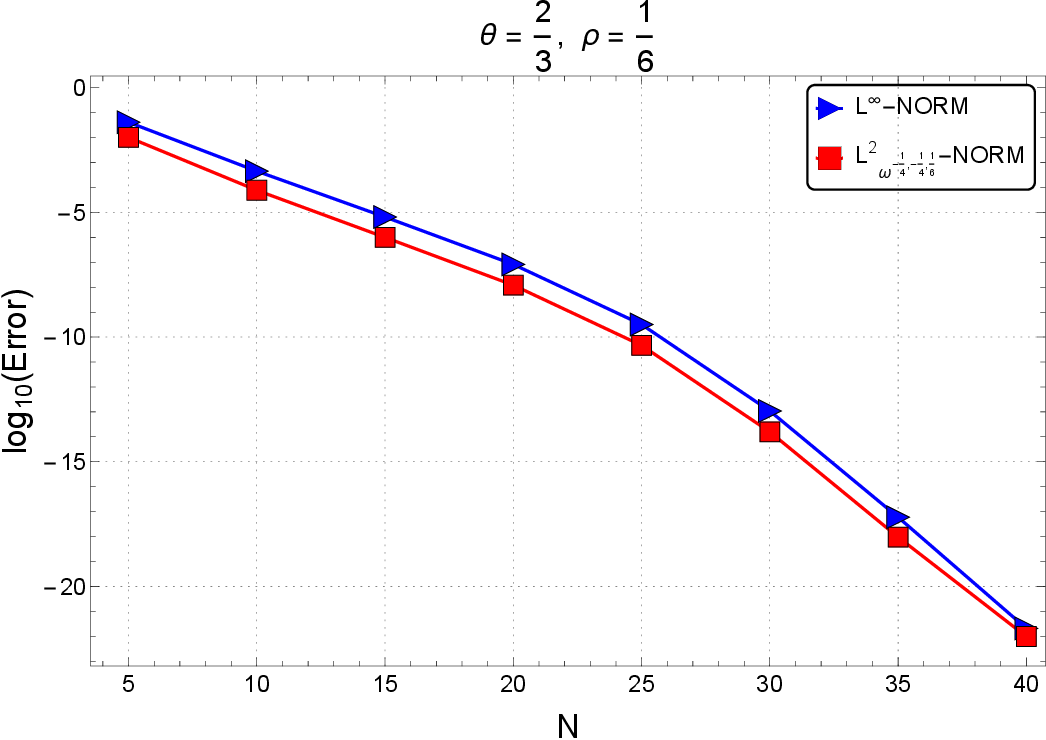}
\caption{
Errors in the \(L^{\infty}(J)\)- and
\(L^{2}_{\varkappa^{-1/4,-1/4,\rho}}(J)\)-norms versus the
\(\rho\)-polynomial degree \(N\):
(a) \(\theta=0.5,\rho=\frac{1}{2}\);
(b) \(\theta=0.5,\rho=\frac{1}{4}\);
(c) \(\theta=\frac{2}{3},\rho=\frac{1}{3}\);
(d) \(\theta=\frac{2}{3},\rho=\frac{1}{6}\).
}
\label{fig:right-sided-example}
\end{figure}
\end{example}

\begin{example}\label{ex:right-sided-test}
We next examine two terminal-singular solutions of
problem~\eqref{eq:3.1R} with \(K(t,\varrho)=1\). The two cases are
\[
\text{Case \rm (i)}\quad
u(t)=(1-t)^{\gamma_1}+(1-t)^{\gamma_2},
\]
and
\[
\text{Case \rm (ii)}\quad
u(t)=\sin\!\left((1-t)^{\gamma_1}+(1-t)^{\gamma_2}\right).
\]
These test problems are used to evaluate the ability of the proposed
fractional backward Zaky--Jacobi spectral method to approximate solutions with
limited regularity near the terminal endpoint \(t=1\).

For arbitrary choices of \(\gamma_1\) and \(\gamma_2\), the transformed
function
\[
u\!\left(1-(1-t)^{1/\rho}\right)
\]
need not be smooth. Nevertheless, its regularity can be described in terms of
the fractional powers appearing in the exact solution. In particular, direct
inspection gives
\[
u\!\left(1-(1-t)^{1/\rho}\right)
\in
B^{2\gamma/\rho+\upsilon+1-\varepsilon}_{\varkappa^{\mu,\upsilon,1}}(J),
\qquad \forall\,\varepsilon>0,
\]
where
\[
\gamma =
\begin{cases}
\infty, & \gamma_1,\gamma_2 \in \mathbb{N},\\[2mm]
\gamma_1, & \gamma_1 \notin \mathbb{N}\ \text{and}\ \gamma_2 \in \mathbb{N},\\[2mm]
\gamma_2, & \gamma_1 \in \mathbb{N}\ \text{and}\ \gamma_2 \notin \mathbb{N},\\[2mm]
\min\{\gamma_1,\gamma_2\}, & \text{in all remaining cases}.
\end{cases}
\]
Here \(B^{s}_{\varkappa^{\mu,\upsilon,1}}(J)\) denotes the weighted Sobolev
space associated with the classical Jacobi weight.

The numerical errors are plotted on a semi-logarithmic scale in
Figs.~\ref{fig:combined-errors-right-sided-case-i} and
\ref{fig:combined-errors-right-sided-case-ii} for several values of
\(\theta\) and \(\rho\), with
\[
\gamma_1=\sqrt{2},\qquad \gamma_2=\sqrt{3}.
\]
Since \(K(t,\varrho)=1\) is smooth, the kernel approximation error in the
theoretical estimates is dominated, for large \(N\), by the approximation
error of the exact solution. Consequently, this experiment primarily
illustrates how the regularity of
\[
u\!\left(1-(1-t)^{1/\rho}\right)
\]
affects the observed convergence rate.
\end{example}

\begin{figure}[h!]
\centering

\begin{subfigure}{0.48\textwidth}
    \centering
    \includegraphics[width=\textwidth]{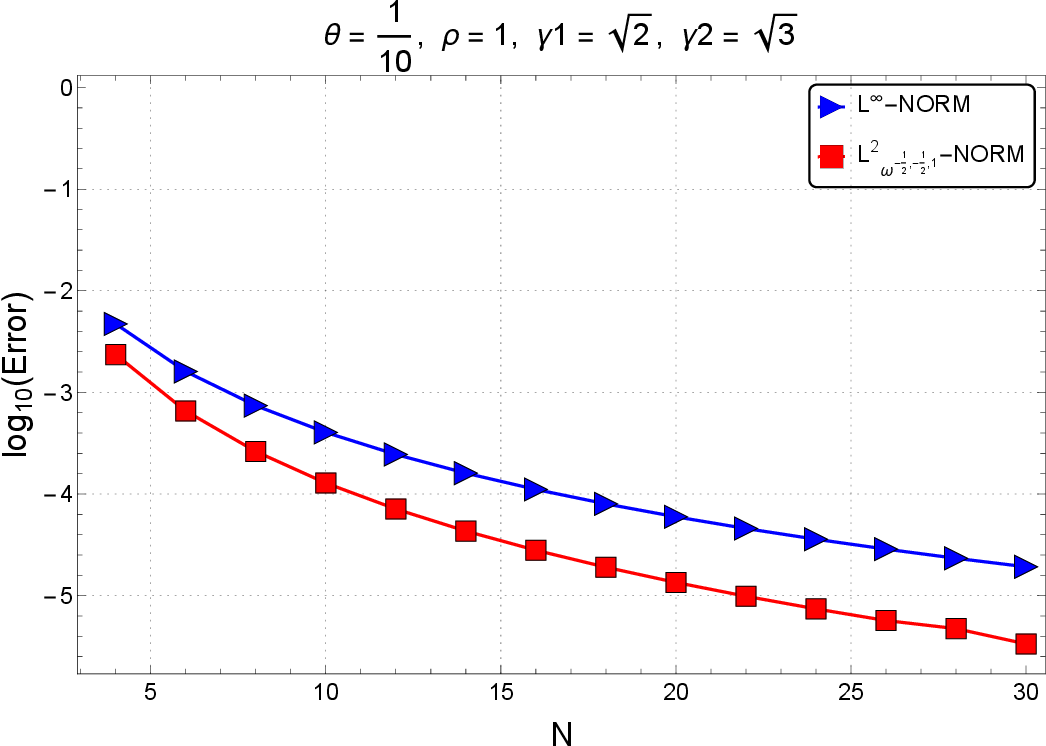}
\end{subfigure}
\hfill
\begin{subfigure}{0.48\textwidth}
    \centering
    \includegraphics[width=\textwidth]{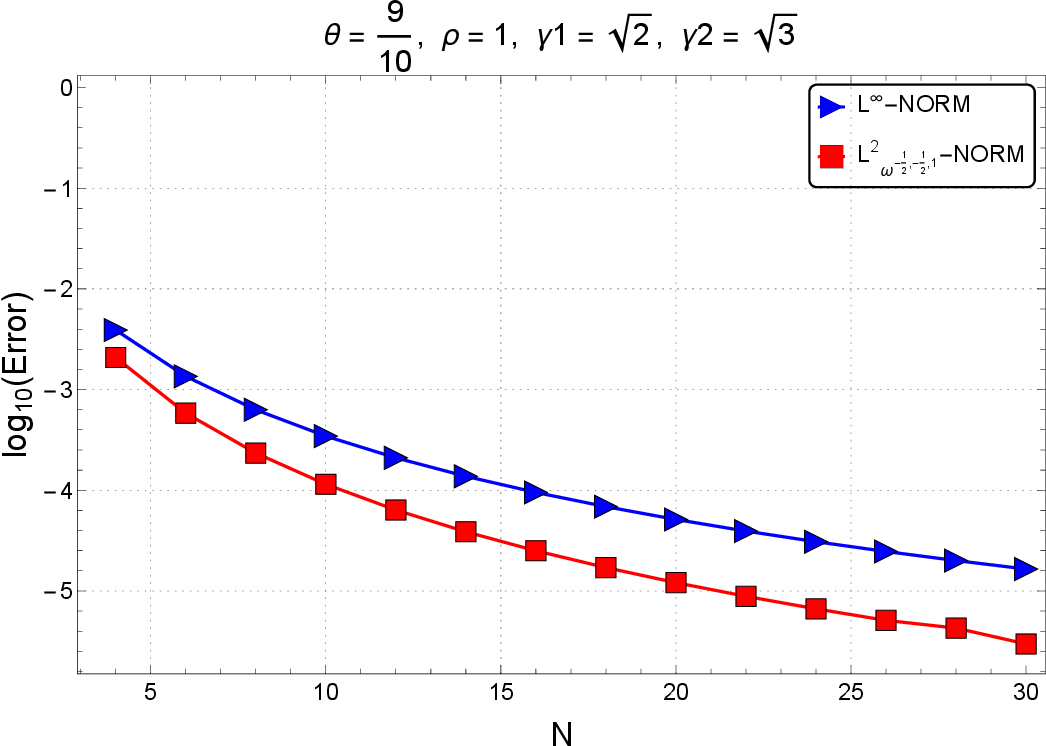}
\end{subfigure}

\vspace{0.35cm}

\begin{subfigure}{0.48\textwidth}
    \centering
    \includegraphics[width=\textwidth]{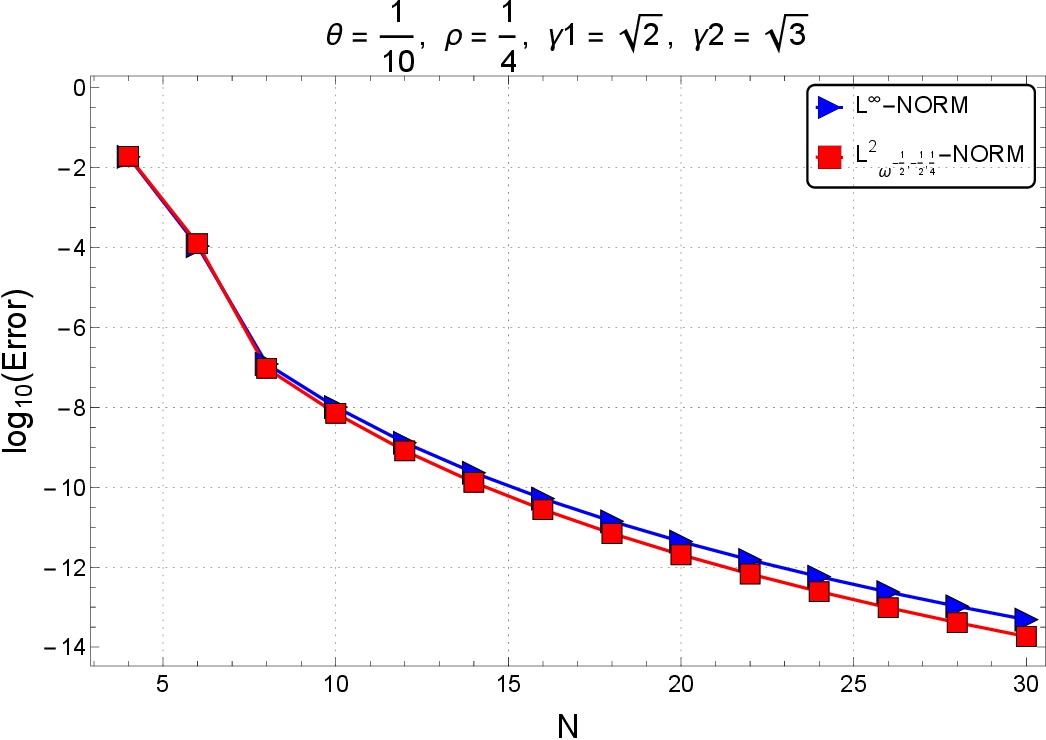}
\end{subfigure}
\hfill
\begin{subfigure}{0.48\textwidth}
    \centering
    \includegraphics[width=\textwidth]{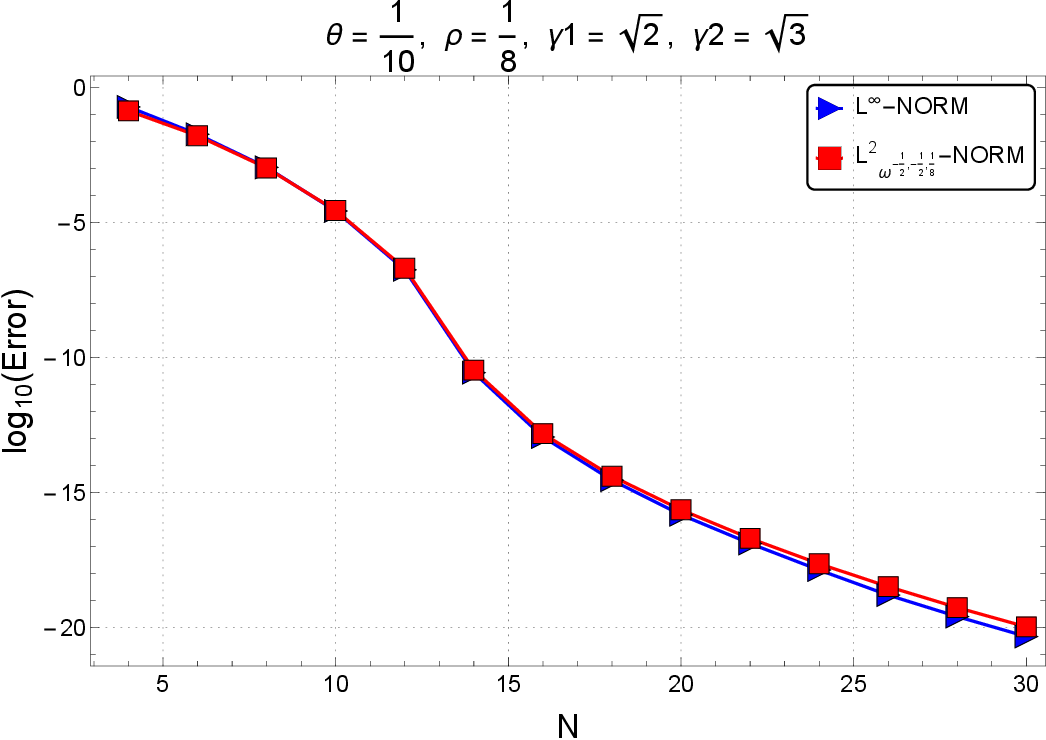}
\end{subfigure}

\vspace{0.35cm}

\begin{subfigure}{0.48\textwidth}
    \centering
    \includegraphics[width=\textwidth]{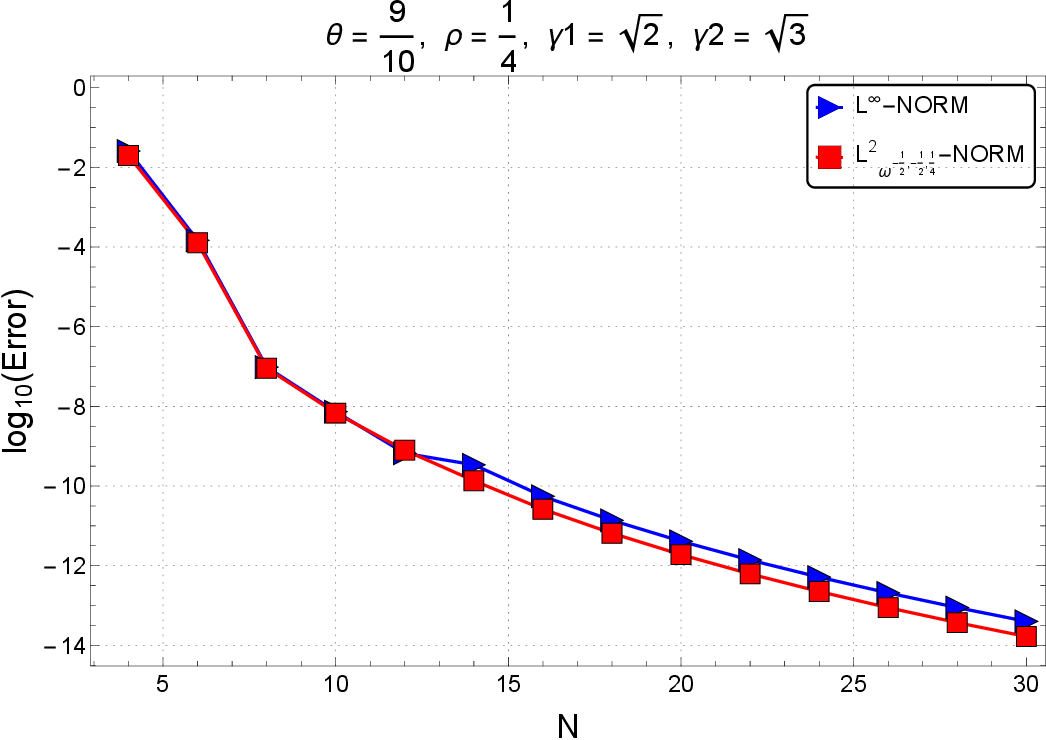}
\end{subfigure}
\hfill
\begin{subfigure}{0.48\textwidth}
    \centering
    \includegraphics[width=\textwidth]{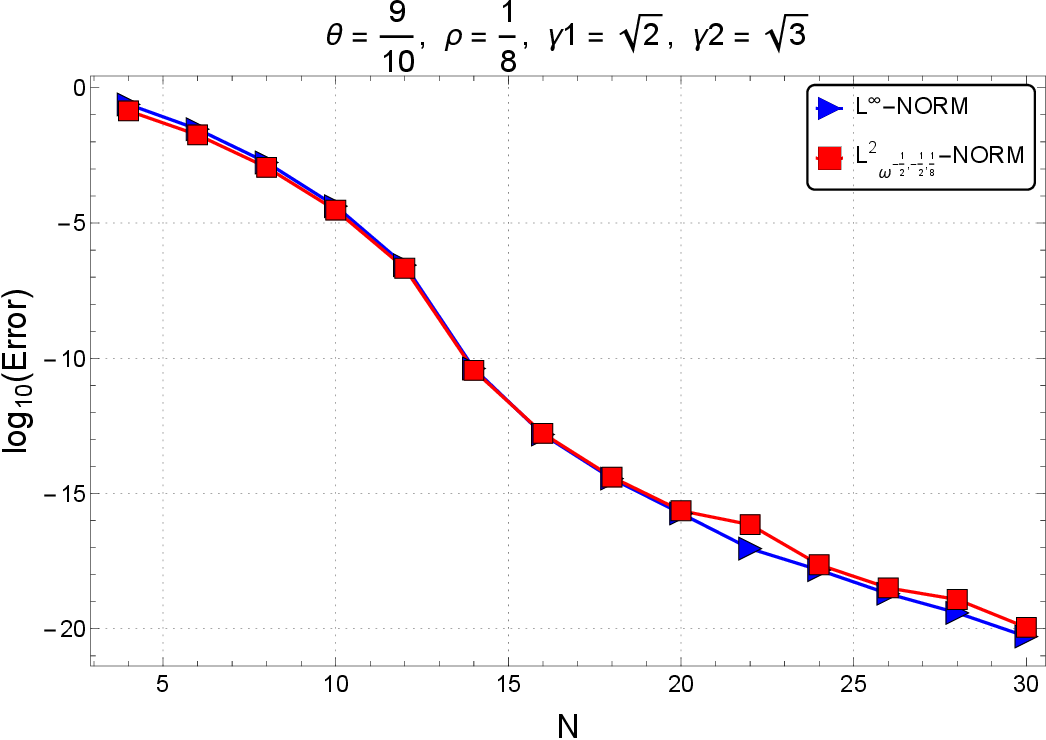}
\end{subfigure}

\caption{
Comparison of the weighted
\(L^{2}_{\varkappa^{0,0,\rho}}(J)\)-error and the
\(L^{\infty}(J)\)-error for the test problem in
case \({\rm (i)}\), with
\(\gamma_{1}=\sqrt{2}\), \(\gamma_{2}=\sqrt{3}\), and different values of
\(\theta\) and \(\rho\).
}
\label{fig:combined-errors-right-sided-case-i}
\end{figure}

\begin{figure}[h!]
\centering

\begin{subfigure}{0.48\textwidth}
    \centering
    \includegraphics[width=\textwidth]{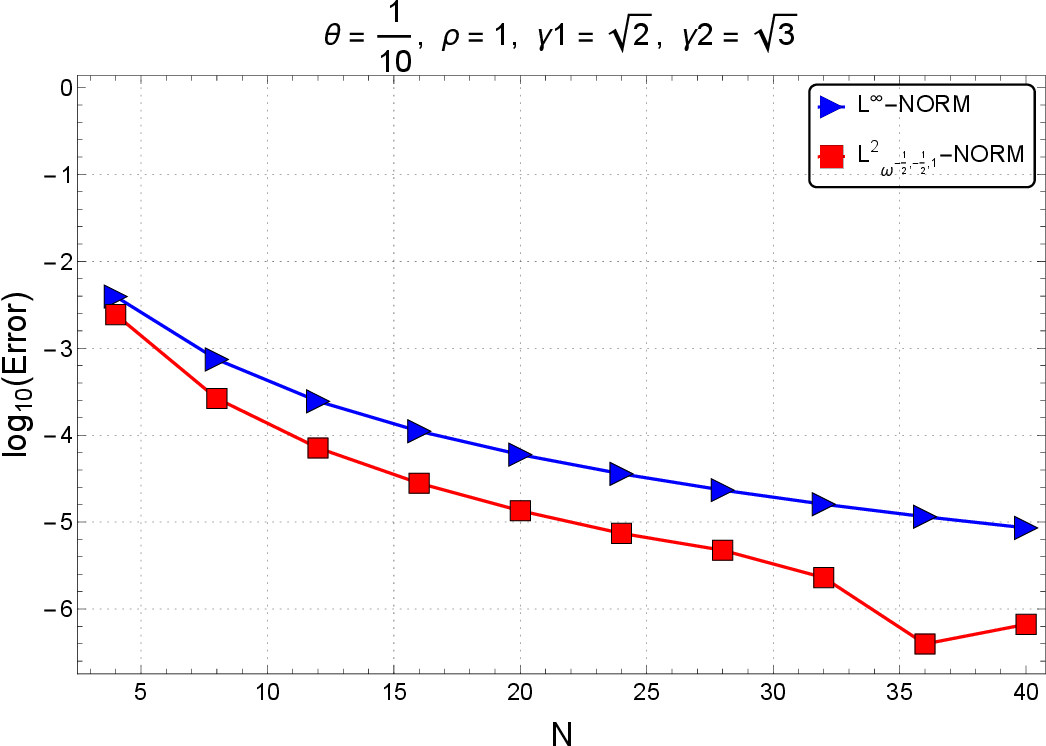}
\end{subfigure}
\hfill
\begin{subfigure}{0.48\textwidth}
    \centering
    \includegraphics[width=\textwidth]{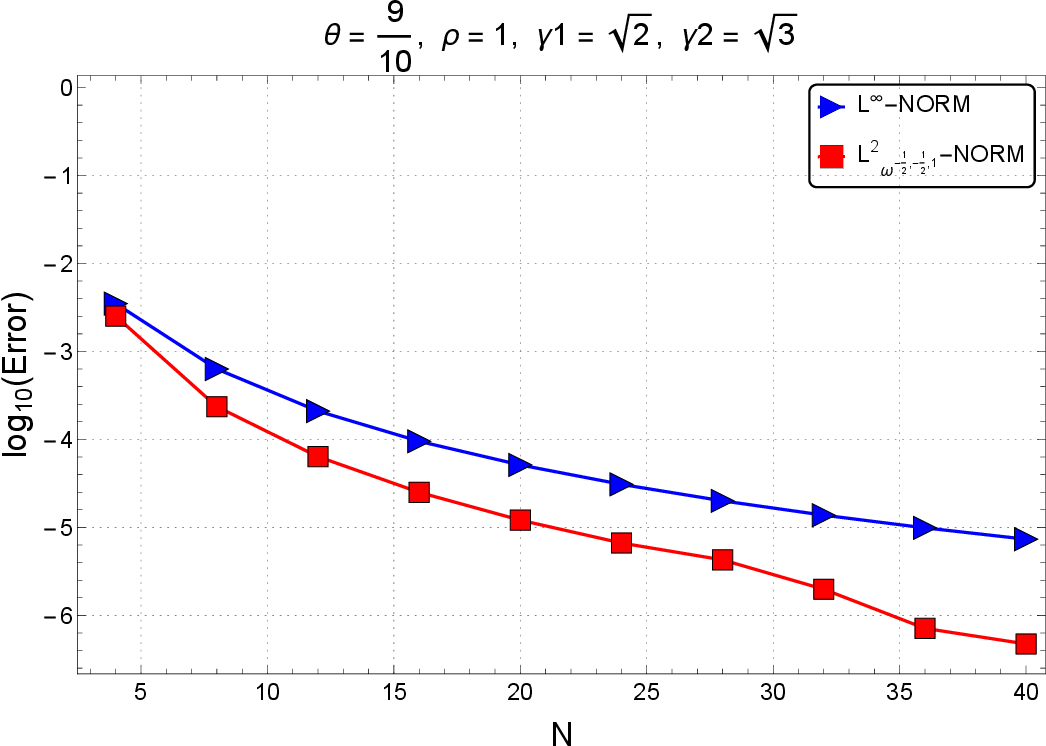}
\end{subfigure}

\vspace{0.35cm}

\begin{subfigure}{0.48\textwidth}
    \centering
    \includegraphics[width=\textwidth]{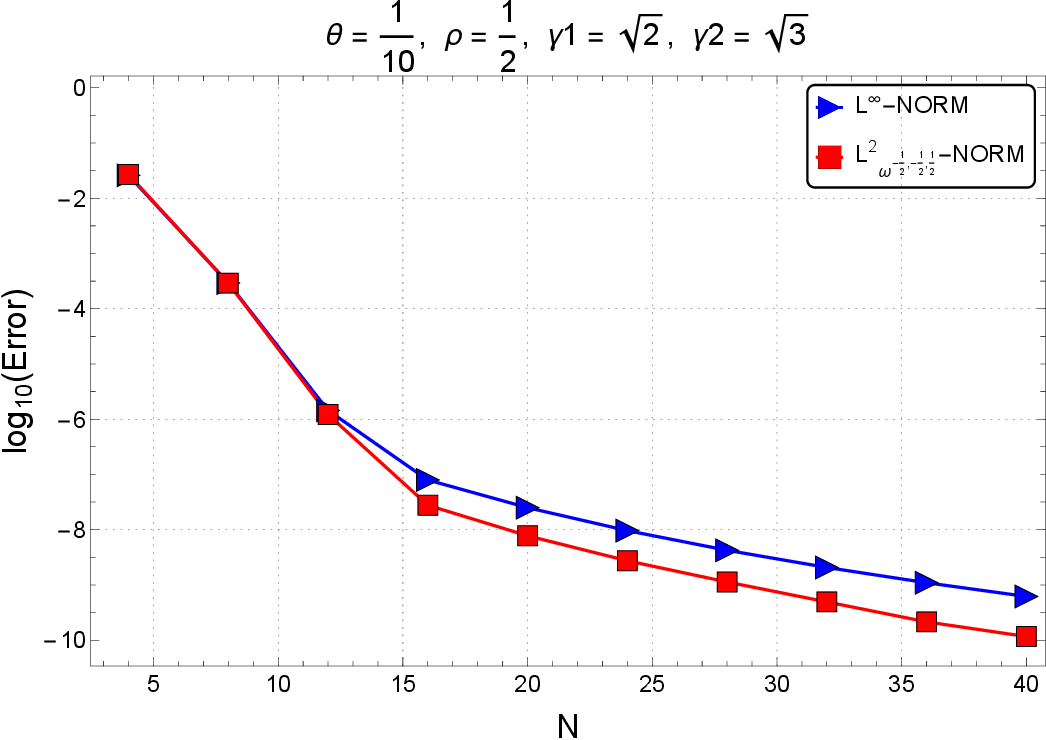}
\end{subfigure}
\hfill
\begin{subfigure}{0.48\textwidth}
    \centering
    \includegraphics[width=\textwidth]{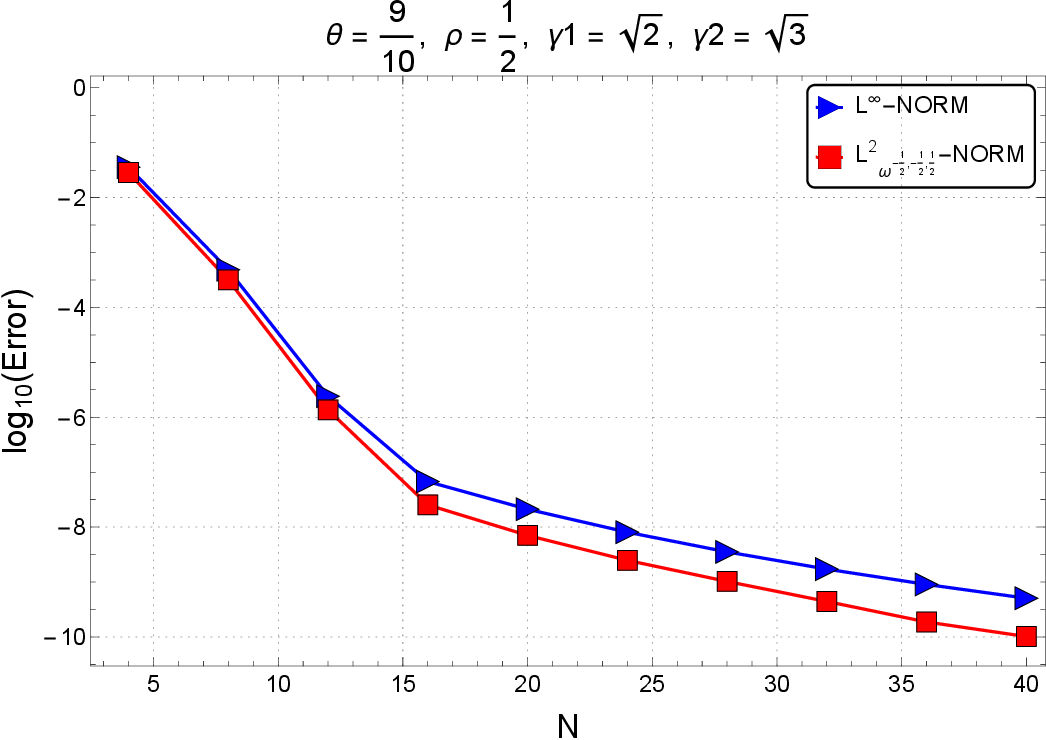}
\end{subfigure}

\vspace{0.35cm}

\begin{subfigure}{0.48\textwidth}
    \centering
    \includegraphics[width=\textwidth]{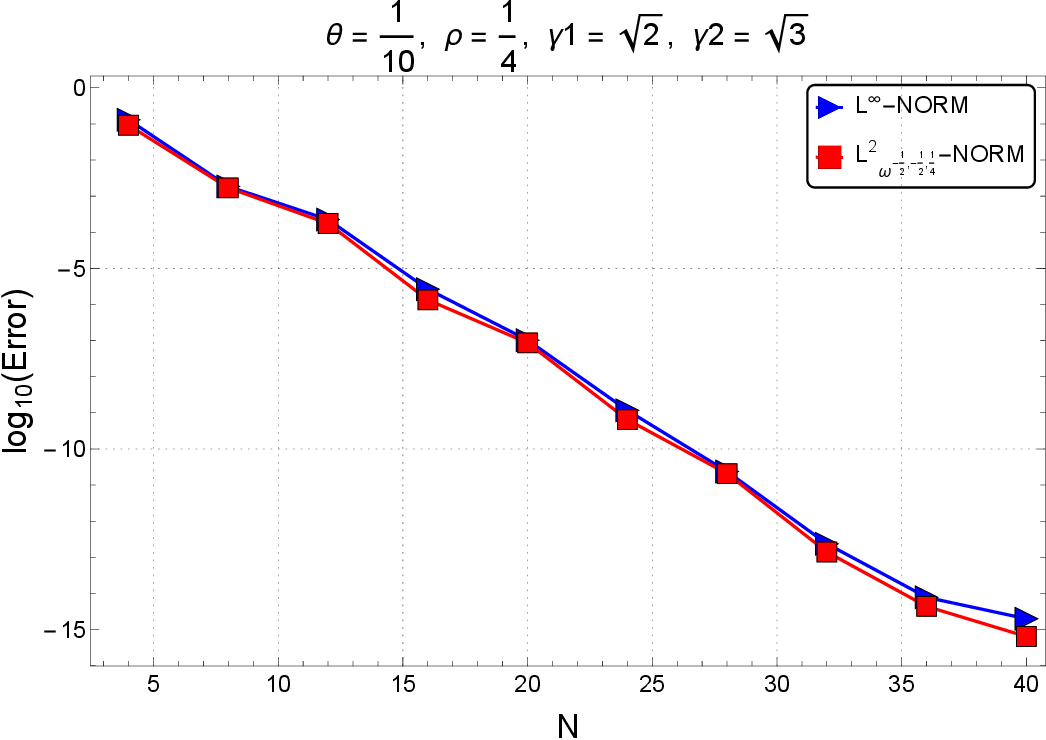}
\end{subfigure}
\hfill
\begin{subfigure}{0.48\textwidth}
    \centering
    \includegraphics[width=\textwidth]{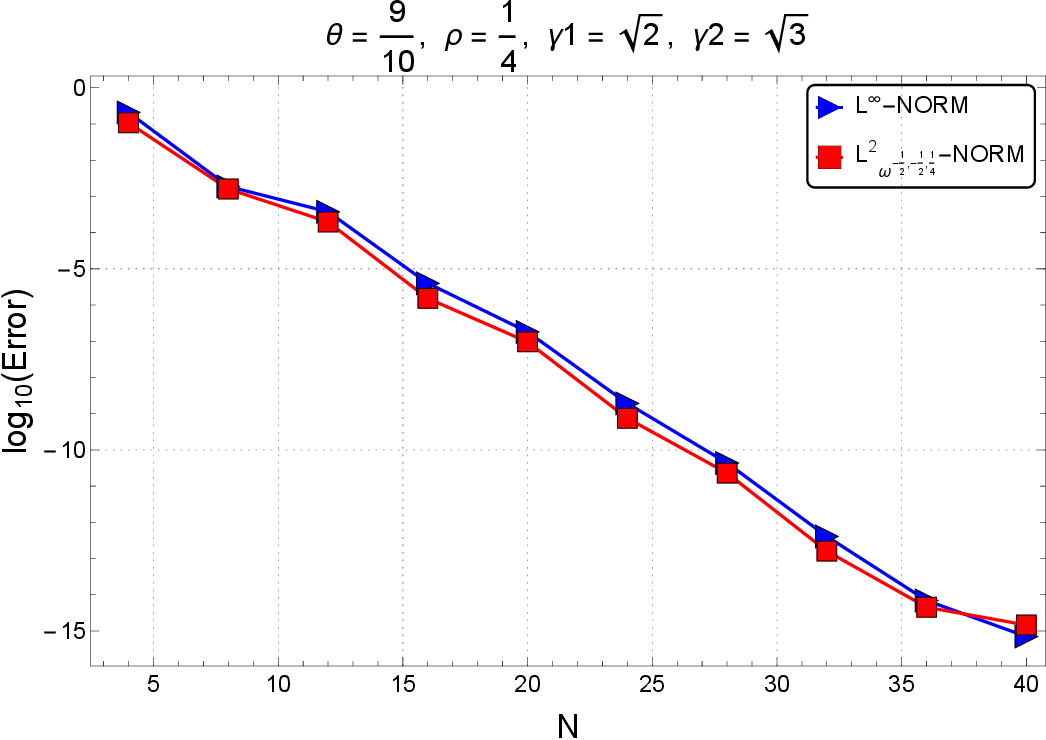}
\end{subfigure}

\caption{
Comparison of the weighted
\(L^{2}_{\varkappa^{0,0,\rho}}(J)\)-error and the
\(L^{\infty}(J)\)-error for the test problem in
case \({\rm (ii)}\), with
\(\gamma_{1}=\sqrt{2}\), \(\gamma_{2}=\sqrt{3}\), and different values of
\(\theta\) and \(\rho\).
}
\label{fig:combined-errors-right-sided-case-ii}
\end{figure}

\section{Conclusion}\label{sec7}

In this paper, we developed a fractional backward Zaky--Jacobi approximation
framework for functions with limited regularity near the terminal endpoint.
The proposed basis is constructed by composing classical Jacobi polynomials
with a backward endpoint algebraic mapping, which enables the approximation
space to capture the intrinsic terminal-endpoint singular structure of the
solution. This feature allows the method to overcome the deterioration in
accuracy typically observed when classical smooth polynomial bases are applied
to weakly regular functions.

A rigorous approximation theory was established for the fractional backward
Zaky--Jacobi functions. We derived their main structural properties, including
orthogonality relations, recurrence formulae, derivative identities, and the
associated singular Sturm--Liouville eigenvalue problem. We also introduced the
corresponding weighted Sobolev spaces generated by the backward transformed
derivative \(\mathcal{D}_{\rho}\), which provide a natural functional setting
for terminal-endpoint singular approximation.

Within this framework, we analyzed the associated
\(L^{2}\)-orthogonal projection and Zaky--Jacobi--Gauss interpolation
operators. Error estimates were proved in weighted norms for both the
approximation error and its transformed derivatives. In addition, inverse
inequalities, weighted Sobolev inequalities, and H\"older-type estimates for
the weakly singular adjoint Volterra integral operator were obtained. These
results form the analytical foundation for the stability and convergence
analysis of the proposed fractional backward spectral-collocation method.

The developed theory was then applied to a weakly singular adjoint Volterra
integral equation. A fully discrete fractional backward spectral-collocation
scheme was constructed by combining the Zaky--Jacobi interpolation operator
with a Jacobi--Gauss quadrature rule adapted to the weak singularity of the
kernel. The resulting error analysis established convergence in both the
\(L^{\infty}(J)\)-norm and the weighted
\(L^{2}_{\varkappa^{\mu,\upsilon,\rho}}(J)\)-norm. The numerical experiments
confirmed the theoretical results and demonstrated the effectiveness of the
proposed method for problems whose solutions exhibit terminal-endpoint weak
singularities.

Overall, the proposed fractional backward Zaky--Jacobi framework provides a
robust and analytically justified basis for high-order approximation of
terminal-endpoint singular functions and for the construction of accurate
spectral-collocation methods for weakly singular adjoint integral equations.

\section*{Data availability statement}
No datasets were generated or analyzed during the current study.

\section*{Declarations}

\section*{Conflict of interest}
The authors declare that they have no conflict of interest.


\bibliographystyle{elsart-num-sort}
		
		
		\bibliography{Bibfileamc}


\end{document}